\documentclass[a4paper]{aptpub}
\usepackage{hyperref}
\usepackage{amsmath,amsfonts,amssymb}
\usepackage{verbatim}
\usepackage{graphicx}
\usepackage{xcolor}
\usepackage{stackengine} 
\authornames{M.~A.~KLATT AND S.~WINTER} 
\shorttitle{Geometric functionals of fractal percolation: convergence a.s.\ and 2nd moments.}

\numberwithin{equation}{section} \numberwithin{thm}{section}

\newcommand{\R}{{\mathbb R}}
\newcommand{\N}{{\mathbb N}}
\newcommand{\Z}{{\mathbb Z}}
\newcommand{\E}{{\mathbb E}}
\newcommand{\BE}{{\mathbb E}}
\renewcommand{\P}{{\mathbb P}}
\newcommand{\deq}{\overset{d}{=}}
\newcommand{\eps}{\varepsilon}
\newcommand{\sS}{\mathcal{S}}
\newcommand{\sT}{\mathcal{T}}
\newcommand{\Sim}{\mathsf{Sim}}

\newcommand{\intr}{\mathsf{int}}

\newcommand{\ind}{\mathbf{1}}

\newcommand{\Var}{\mathrm{Var}} 
\newcommand{\Cov}{\mathrm{Cov}}

\newcommand{\osV}{\overline{\mathcal V}}

\stackMath
\newcommand\widecheck[1]{%
\savestack{\tmpbox}{\stretchto{%
  \scaleto{%
    \scalerel*[\widthof{\ensuremath{#1}}]{\kern-.6pt\bigwedge\kern-.6pt}%
    {\rule[-\textheight/2]{1ex}{\textheight}}
  }{\textheight}%
}{0.5ex}}%
\stackon[1pt]{#1}{\scalebox{-1}{\tmpbox}}%
}

\begin{document}

\title{Almost sure convergence and second moments of geometric functionals of fractal percolation.}

\authorone[Heinrich-Heine-Universit\"at D\"usseldorf]{Michael A.\ Klatt} 
\addressone{Institut f\"ur Theoretische Physik II: Weiche Materie,
Heinrich-Heine-Universit\"at D\"usseldorf, 40225 D\"usseldorf,
Germany; Experimental Physics, Saarland University, Center for
Biophysics, 66123 Saarbrücken, Germany;
Present address:
Institut für KI Sicherheit, Deutsches Zentrum für Luft‐ und Raumfahrt (DLR), Wilhelm‐Runge‐Straße 10, 89081 Ulm, Germany;
Institut für Materialphysik im Weltraum, Deutsches Zentrum für Luft- und Raumfahrt (DLR), 51170 Köln, Germany;
Department of Physics, Ludwig-Maximilians-Universität, Schellingstraße 4, 80799 Munich, Germany.
} 
\emailone{michael.klatt@dlr.de} 

\authortwo[Karlsruhe Institute of Technology]{Steffen Winter}
\addresstwo{Karlsruhe Institute of Technology, Department of Mathematics, Englerstr.~2, 76128 Karlsruhe, Germany}
\emailtwo{steffen.winter@kit.edu}
%
%
%
%



\begin{abstract}
We determine almost sure limits of rescaled intrinsic volumes of the construction steps of fractal percolation in $\R^d$ for any dimension $d\geq 1$.
We observe a factorization of these limit variables which allows, in particular, to determine their expectations and covariance structure. We also show convergence of rescaled expectations and variances of the intrinsic volumes of the construction steps to expectations and variances of the limit variables and give rates for this convergence in some cases.
These results significantly extend our previous work that addressed only limits of expectations of intrinsic volumes.
\end{abstract}


\section{Introduction} \label{sec:intro}

Let $p\in(0,1]$ and $M\in\N$, $M\geq 2$. Divide the unit cube $J:=[0,1]^d$ of $\R^d$ into $M^d$ subcubes of sidelength $1/M$. 
Keep each of these subcubes independently with probability $p$ and discard it otherwise. 
Then repeat this construction independently for each of the retained cubes, producing a random collection of cubes of sidelength $1/M^2$, and so on. 
For $n\in\N$, denote by $F_n$ the union of the cubes retained in the $n$-th step of this construction.
The sets $F_n$ form a decreasing sequence of random compact subsets of $J$ and its limit
\begin{align}
  F:= \bigcap_{n\in\N_0} F_{n}
\end{align}
is known as \emph{fractal percolation} or \emph{Mandelbrot percolation}, cf.\ e.g.\
\cite{Mandelbrot74,CCD88}. It is easily seen that for any $p<1$, $F$ has a positive probability of being empty and it is well known that $F$ is in fact almost surely empty if $p$ is too small, i.e., if $p\leq 1/M^d$.
For $p>1/M^d$, however, there is a positive probability (depending on $p,M$ and $d$) that $F\neq \emptyset$, and conditioned on $F$ being nonempty,  the Hausdorff dimension and equally the Minkowski dimension of $F$ are almost surely given by the number
\begin{align} \label{eq:dimF}
   D:=\frac{\log(M^d p)}{\log(M)},
\end{align}
see e.g.\ \cite{CCD88}. Many properties of this simple model have been studied, including e.g.\ its connectivity \cite{CCD88,BroCam08,BroCam10}, its visibility (or behaviour under projections) \cite{MR2928497,MR3316924,MR3163542}, its porosity \cite{BJ19,MR3736180}, path properties \cite{MR1378847,BCJM13} and very recently its (un-)rectifiability \cite{MR4291467}.

Intrinsic volumes are a basic tool in stochastic geometry and other fields of mathematics providing essential geometric information about complex structures. We refer to \cite[Ch.~4]{Schneider14} or \cite[Ch.~14.2]{SchneiderWeil08} for their definition and properties, see also the summary in our previous paper \cite[p.~1087-1088]{KW18}. Let $V_k(F_n)$, $k=0,1,\ldots, d$ denote the intrinsic volumes of the random set $F_n$ in $\R^d$. Note that they are well defined, since each $F_n$ is a finite union of cubes and thus polyconvex almost surely. In \cite{KW18}, we have studied the expected intrinsic volumes $\E V_k(F_n)$ of the construction steps $F_n$ and in particular their limiting behaviour as $n\to\infty$. More precisely, we have established the existence of the limit functionals 
$$
  \osV_k(F):=\lim_{n\to\infty} M^{n(k-D)} \E V_k(F_{n})
  $$
and derived formulas to compute them, see also Theorem~\ref{thm:Vk-limit-general} below. Moreover,
we have investigated possible connections of these functionals with percolation thresholds, see \cite{KW18}.

In this article we continue our investigation of the random variables $V_k(F_n)$ and their limiting behaviour as $n\to\infty$. We are interested in the question which further properties of these random geometric functionals beyond convergence of expectations can be derived.  Here we will prove in particular that the random variables $V_k(F_n)$, when properly rescaled, converge almost surely, as $n\to\infty$, to some nontrivial limit, and we will determine expectations and covariances of the limit variables. Moreover, for $k=d$ and $d-1$ (i.e., for volume and surface area) we derive expansions for the expectations and variances of the functionals $V_k(F_n)$ of the $n$-th approximations. This will not only enable us to show convergence of these expectations and variances (when suitably rescaled) as $n\to\infty$ to expectation and variance of the limit variable but also to obtain the rates for this convergence.

{\bf Outline.} In the next two sections we will formulate our main results. The almost sure convergence of the rescaled intrinsic volumes and some consequences are discussed in Section~\ref{sec:main}, while some additional results regarding the finite approximations are addressed in Section~\ref{sec:additional}. In order to prepare the proofs, we review in Section~\ref{sec:2}  random iterated function systems as well as Nerman's renewal theorem for branching random walks. 
Sections~\ref{sec:proof}--\ref{sec:var_surf} are devoted to the proofs. In Section~\ref{sec:proof} the main result of Section~\ref{sec:main} is proved, and in Section~\ref{sec:var_vol} expectation and variance of the volume of the finite approximations $F_n$ are discussed. Finally, in Section~\ref{sec:var_surf}, expectation and variance of the surface area of $F_n$ are addressed.

\section{Almost sure convergence of rescaled intrinsic volumes} \label{sec:main}

Let $F$ 
be a fractal percolation on $J=[0,1]^d$ with parameters $M\in\N_{\geq 2}$ and $p\in(0,1]$. Recall that the $n$-th construction step $F_n$ is a random union of cubes of sidelength $1/M^n$, which we will call the \emph{basic cubes of level $n$} in the sequel. 
Let us first focus on the (random) numbers $N_n$ of basic cubes of level $n$ contained in $F_n$. For convenience, we also set $F_0:=J$ to be the unit cube and  $N_0:=1$. The sequence of the random variables $N_n$, $n\in\N_0$ forms a Galton-Watson process with a binomial offspring distribution with parameters $M^d$ and $p$. Indeed, by the assumed independence in the subdivision and retention procedure, the number of preserved subcubes of any existing cube of any level is a sum of $M^d$ independent Bernoulli random variables with parameter $p$ and thus $\text{Bin}(M^d,p)$-distributed.  
In particular, we have $N_1\sim\text{Bin}(M^d,p)$ and hence $\E N_1=M^dp$ and $\Var(N_1)=M^dp(1-p)$. From this, one can easily deduce (see e.g.~\cite[Ch. I.A.2]{AN72}) that mean and variance of $N_n$ are given by $\E N_n=(\E N_1)^n=(M^dp)^n$ and
\begin{align} \label{eq:variance-N-n}
  \Var(N_n)&=\frac{\Var(N_1)}{\E N_1-1} (\E N_1)^{n-1}((\E N_1)^n-1) \notag\\
  &=\frac{1-p}{M^dp-1} (M^dp)^n\left((M^dp)^n-1\right),
\end{align}
provided $\E N_1\neq 1$, and $\Var(N_n)=n\cdot \Var(N_1)$ in case $\E N_1=1$, i.e.\ $p=1/M^d$.

Furthermore, it is well known (see e.g.~~\cite[Thm.~I.6.1, p.9]{AN72}) that the sequence $W_n:=N_n/\BE N_n$, $n\in\N$ is a martingale with respect to the filtration naturally induced by the construction steps of the process (i.e., the sequence $(\mathcal{F}_n)_n$, where $\mathcal{F}_n$ is the $\sigma$-algebra generated by $N_0,N_1,\ldots,N_n$). Since $W_n\geq 0$, this implies that there exists a random variable $W_\infty$ such that
\begin{align}\label{eq:def-W_infty}
\lim_{n\to\infty} W_n=W_\infty \qquad \text{ almost surely. }
\end{align}
Moreover, for $p>M^{-d}$, we have $\E N_1>1$ (and $\Var(N_1)<\infty$) and therefore \cite[Thm.~I.6.2]{AN72} implies
\begin{align} \label{eq:var_W_infty}
  \E W_\infty=1 \quad \text{ and } \quad \Var(W_\infty)=\frac{\Var(N_1)}{(\E N_1)^2-\E N_1}=\frac{1-p}{M^dp-1}.
\end{align}


In \cite{KW18}, we studied the limiting behaviour of expected intrinsic volumes $\E V_k(F_n)$ of the construction steps $F_n$, as $n\to\infty$, and found that these functionals converge for all parameter combinations $M,p$ when suitably rescaled. We recall the main result concerning this convergence. Denote the basic cubes of level 1 by $J_1,\ldots, J_{M^d}$ and let, for each $j\in\Sigma:=\{1,\ldots,M^d\}$ and each $n\in\N$,
$
F_n^j
$
be the union of those basic cubes of level $n$ that are contained in the union $F_n$ and subcubes of $J_j$, see also \eqref{eq:Fj_ndef} for a formal definition. 
\begin{thm}\label{thm:Vk-limit-general} {\cite[Theorem 1.1]{KW18}}
Let $F$ 
be a fractal percolation on $[0,1]^d$ with parameters $M\in\N_{\geq 2}$ and $p\in(0,1]$. 
Let $D$ be as in~\eqref{eq:dimF}. 
  Then, for each $k\in\{0,\ldots,d\}$, the limit
  $$
  \osV_k(F):=\lim_{n\to\infty} M^{n(k-D)} \E V_k(F_{n})
  $$
  exists and is given by the expression
  \begin{align} \label{eq:Vk-limit-general}
    V_k([0,1]^d) +\sum_{T\subset\Sigma,|T|\geq 2}(-1)^{|T|-1} \sum_{n=1}^\infty M^{n(k-D)} \E V_k(\bigcap_{j\in T} F^j_{n}).
  \end{align}
\end{thm}

Thus, the expectations $\E V_k(F_n)$ of the random variables $V_k(F_n)$ converge as $n\to\infty$, when rescaled with the sequence $M^{n(k-D)}$. It is a natural question, whether also the random variables $V_k(F_n)$ themselves converge. In view of the convergence behaviour of their expectations, it is likely to be a good idea to rescale them in the same way by $M^{n(k-D)}$.  Therefore, we define the random variables
\begin{align*}
  Z^k_n:=M^{(k-D)n} V_k(F_n), \qquad n\in\N, \quad k\in\{0,1,\ldots,d\}.
\end{align*}
The following statement is our main result. It establishes for each $k\in\{0,\ldots,d\}$ the almost sure convergence of the sequence $(Z^k_n)$ as $n\to\infty$ for all parameter combinations $(p,M)$ for which a nontrivial limit set $F$ exists, i.e.\ whenever $p>M^{-d}$.

\begin{thm} \label{thm:as-conv-Vk}
Let $F$ 
be a fractal percolation on $[0,1]^d$ with parameters $M\in\N_{\geq 2}$ and $p\in(M^{-d},1]$. 
  Then, for each $k\in\{0,\ldots,d\}$, the random variables $Z^k_n$ converge almost surely, as $n\to\infty$, to some random variable $Z^k_\infty$. Moreover, $Z^k_\infty$ is given by
  \begin{align} \label{eq:factorization}
     Z^k_\infty=\osV_k(F)\cdot W_\infty,
  \end{align}
  i.e., $Z^k_\infty$ factorizes into a deterministic part $\osV_k(F)$, given by Theorem~\ref{thm:Vk-limit-general}, and a random part $W_\infty$, which is the martingale limit given by \eqref{eq:def-W_infty}.
\end{thm}

The factorization of the limit variable $Z^k_\infty$ in \eqref{eq:factorization} separates the probabilistic effects from the geometric information. The random variations in the limit set $F$  only depend on the underlying Galton-Watson process, i.e., on the number of cubes that survive but not on their positions or mutual intersections. Such geometric information is captured by the purely deterministic factors $\osV_k(F)$, which
depend on $k$ and describe some almost sure geometric property of the
limit set $F$.


Fortunately, expectation and variance of $W_\infty$ are well known, see \eqref{eq:var_W_infty}. This allows to derive from the above statement the complete covariance structure of the random variables $Z^0_\infty, Z^1_\infty,\ldots,Z^d_\infty$.

 \begin{cor}
    For each $k\in\{0,\ldots,d\}$, the limit variable $Z^k_\infty$ has mean $\osV_k(F)$ and positive and finite variance given by
  \begin{align}
    \Var(Z^k_\infty)=(\osV_k(F))^2\cdot \Var(W_\infty)= (\osV_k(F))^2\cdot\frac{1-p}{M^d p-1}.
    \label{eq:var}
  \end{align}
  Moreover, for any $k,\ell\in\{0,\ldots,d\}$,
  \begin{align}
    \Cov(Z^k_\infty,Z^\ell_\infty)=\osV_k(F)\cdot \osV_{\ell}(F)\cdot\Var(W_\infty)= \osV_k(F)\cdot \osV_{\ell}(F)\cdot\frac{1-p}{M^d p-1}.
    \label{eq:cov}
  \end{align}
 \end{cor}
 \begin{proof}
   This follows immediately by combining \eqref{eq:factorization} with \eqref{eq:var_W_infty}.
 \end{proof}

 \begin{figure}[t]
   \centering
   \includegraphics[width=0.45\textwidth]{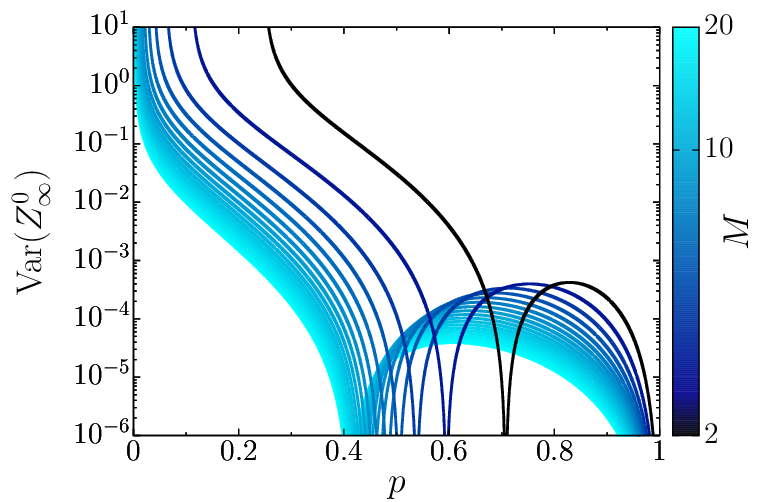}
   \hfill%
   \includegraphics[width=0.45\textwidth]{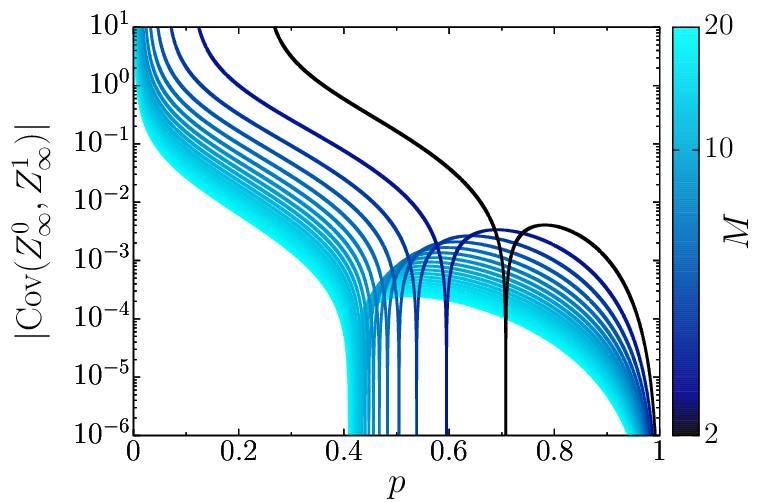}\\
   \includegraphics[width=0.45\textwidth]{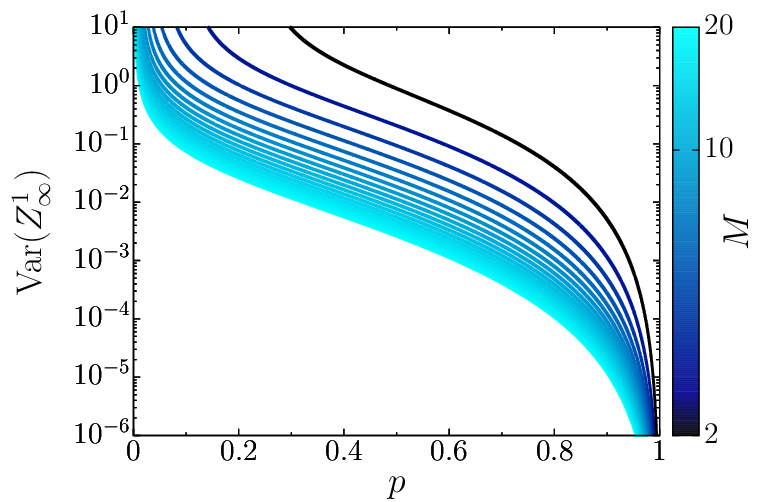}
   \hfill%
   \includegraphics[width=0.45\textwidth]{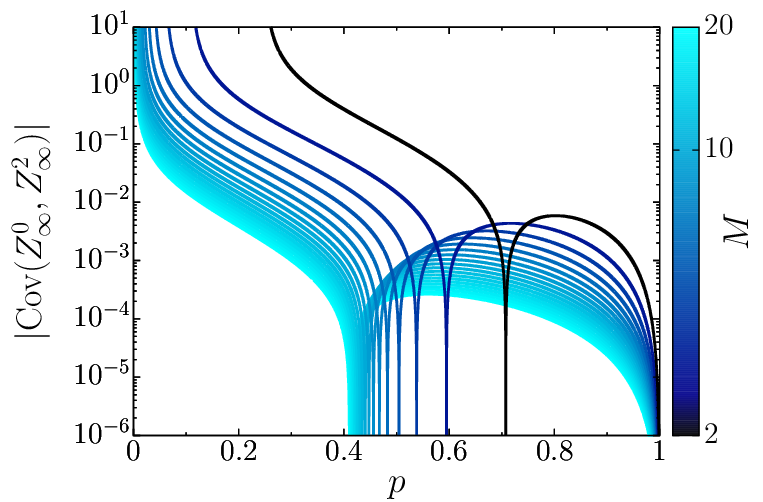}\\
   \includegraphics[width=0.45\textwidth]{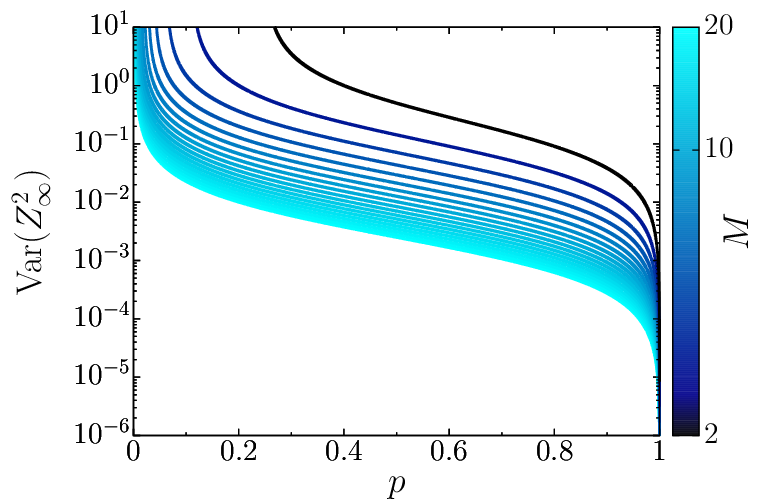}
   \hfill%
   \includegraphics[width=0.45\textwidth]{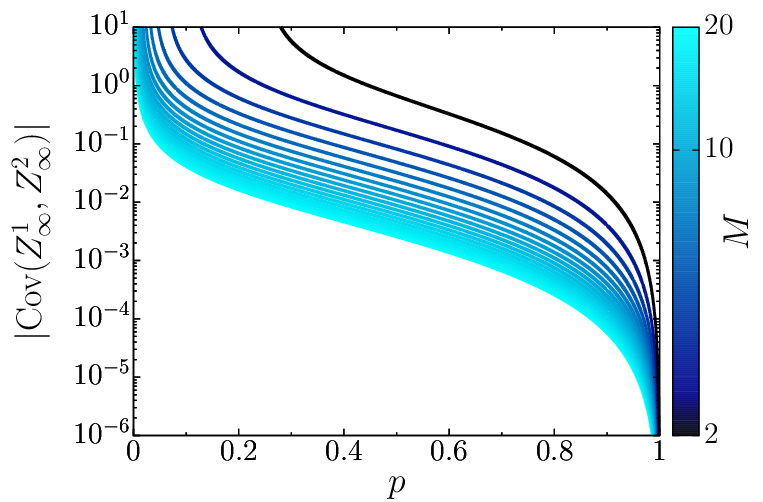}
   \caption{Semi-logarithmic plots of the variances and the absolute values
   of the covariances of $Z^k_{\infty}$ in $\R^2$ as functions of $p$ for different values of $M\leq 20$
   (indicated by the color scale), see \eqref{eq:var} and \eqref{eq:cov}. For $M\to\infty$, all variances and covariances converge to zero.}
   \label{fig:cov}
 \end{figure}

 For plots of variances and covariances of $Z^k_{\infty}$, see
 Fig.~\ref{fig:cov}.
 The proof of Theorem~\ref{thm:as-conv-Vk} is given in Section~\ref{sec:proof} below. We will employ a renewal theorem for branching random walks due to Nerman \cite{Nerman81}, which has been used before in the setting of random self-similar sets by Gatzouras \cite{Gatzouras00} and Z\"ahle~\cite{Z11}. In these papers the studied functionals depend on a continuous parameter (the radius of a parallel set grown around the limit set $F$) which is sent to zero, while in our case the approximation of $F$ is by a discrete sequence, for which we derive below a variant of the renewal theorem, cf.\ Proposition~\ref{prop:RTdiscrete}. While a similar factorization as in Theorem~\ref{thm:as-conv-Vk} has been observed in these previous papers, our somewhat simpler setting allows for explicit computations of the limits. Note also that we do not need any additional integrability or regularity conditions, such as the ones imposed e.g.\ in Z\"ahle~\cite{Z11}.

 \begin{rem}
 From the factorization in Theorem~\ref{thm:as-conv-Vk} it is clear that further progress concerning the distributional properties of the limit variables $Z^k_\infty$ solely depends on understanding the distribution of $W_\infty$, as the expectations $\osV_k(F)$ are known already, cf.\ Theorem~\ref{thm:Vk-limit-general} and \cite{KW18} for more explicit expressions in $\R^2$ and $\R^3$.
  It is clear that $\P(W_\infty=0)=\P(F=\emptyset)$ is strictly positive for all $p\neq 1$, meaning that the distribution of $W_\infty$ (and thus of the $Z^k_\infty$) has an atom at $0$.
 From the more advanced theory of branching processes, it follows that the distribution of $W_\infty$ is absolutely continuous on the open interval $(0,\infty)$, see \cite[Corollary I.12.1]{AN72}.
 \end{rem}

 \begin{rem}
  In \cite{KW18} we have addressed the question, whether percolation thresholds can be approximated by the roots or critical points of mean intrinsic volumes $\osV_k(F_p)$ (as functions of $p$).
The question was motivated by observations in \cite{NMW08,KSM17} for many classical percolation models. Similar questions have been asked for second moments.
Indeed, Last and Ochsenreither \cite{LO14} observed that the variance of the expected Euler characteristic of Poisson-Voronoi percolation in the plane has a maximum at $\frac 12$, which equals the percolation threshold in this model.
Simulations in \cite{KSM17} suggest that for various variants of continuum percolation in $\R^2$ (Boolean models of rectangles) some minimum of the variance of the Euler characteristic as well as of the covariance of Euler characteristic and perimeter are good approximations for the percolation thresholds in these models.
For fractal percolation in the plane we observe that (for any $M$) the local maximum of the variance of the expected Euler characteristic (as a function of $p$) may provide a rather close lower bound for the (unknown) percolation threshold $p_c(M)$, but we have not tried to prove it. Recall from  \cite{CC89} that, as $M\to\infty$, $p_c(M)\searrow p_{c,NN}$, the threshold of nearest-neighbour site percolation. A comparison to the value $p_{c,NN}\approx 0.59274621(13)$ provided by simulations, see e.g.~\cite{NZ00}, indicates that for large $M$ this local maximum could be a tight lower bound for $p_c(M)$, cf.~Fig.~\ref{fig:2}.
However, there are at least two arguments to the contrary to such a close relation: (1) the observed factorization suggests that there is no additional geometric information in the 2nd moments; (2) as discussed in \cite[end of Sect.~2]{KW18} (based on a result in \cite{BCJM13}), the \emph{dust} (i.e., the points not contained in large clusters) may have a dominant influence on the functionals $\osV_k(F_p)$. That is, it is possible that in the supercritical regime these functionals do not see the \emph{backbone} but only the dust.
 \end{rem}

 \begin{figure}[t]
   \centering
   \includegraphics[width=\textwidth]{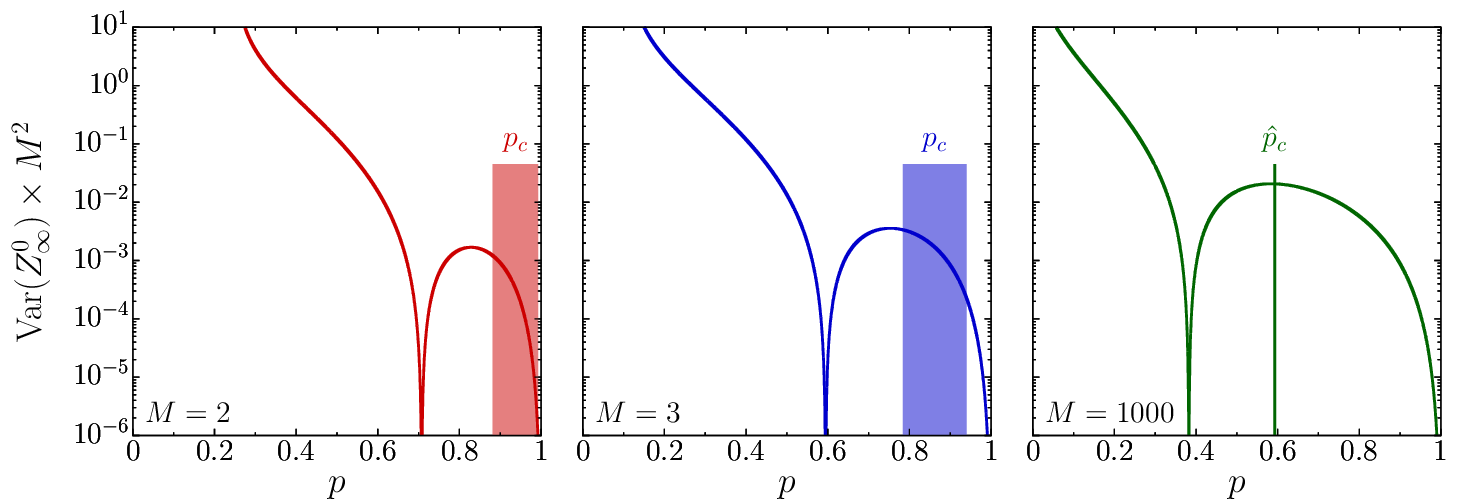}
   \caption{Semi-logarithmic plots of the variance of $Z^0_{\infty}$ rescaled by $M^2$ as functions of $p$ for $M=2$ (left), 3 (center), and 1000 (right). The shaded areas for $M=2$ and 3 indicate rigorously known bounds on the percolation threshold~\cite{Don15} as in \cite[Fig.~4]{KW18}. The vertical line for $M=1000$ indicates an empirical estimate of $p_{c,NN}$ provided by simulations~\cite{NZ00}. }
   \label{fig:2}
 \end{figure}

\section{Additional results: Analysis of the finite approximations} \label{sec:additional}

A priori it is not clear whether the almost sure convergence $Z^k_n\to Z^k_\infty$ obtained in Theorem~\ref{thm:as-conv-Vk} implies the convergence of the expectations $\E Z^k_n$  to the expectation $\E Z^k_\infty$ of the limit variable, as $n\to\infty$, or whether any higher moments converge. However, for the expectations, this convergence can easily be deduced from the above results. Indeed, by Theorem~\ref{thm:Vk-limit-general}, $\E Z^k_n\to \osV_k(F)$ and, by Corollary~\ref{cor:l-indep-fps}, the latter number equals $\E Z^k_\infty$.
In \cite{KW18}, we observed that the convergence, as $n\to\infty$, of the rescaled expected intrinsic volumes $M^{(k-D)n} \E V_k(F_n)=Z^k_n$ to the limit functionals $\osV_k(F)$ is very fast, more precisely their difference decays exponentially. In \cite[Remark~5.2]{KW18} we have explicitly determined the speed of convergence in some cases, which can now be interpreted as the speed of the convergence $\E Z^k_n\to \E Z^k_\infty$. For $d=2$ and $k=0$, for instance, one has $\E Z^0_n-\E Z^0_\infty\sim c (p/M)^n$ as $n\to\infty$ (where $c$ is explicitely known).  But what about higher moments? In Corollary~\ref{cor:l-indep-fps} we have obtained the variances of the limit variables $Z^k_{\infty}$. But do the variances $\Var(Z^k_n)$ converge to $\Var(Z^k_\infty)$? We provide here some explicit expressions for the variances of the $n$-th approximations for some of the functionals, namely for volume and surface area. They allow to deduce that also the variances of the rescaled functionals $Z^k_n=M^{(k-D)n} V_k(F_n)$ converge exponentially fast to the variance of the limit variable, and we provide explicit rates.
We start with the volume.

 \begin{thm} \label{thm:var-vol}
Let $F$ 
be a fractal percolation on $[0,1]^d$ with parameters $M\in\N_{\geq 2}$ and $p\in(0,1]$. 
  Then, for each $n\in\N$, $\E V_d(F_n)=p^{n}$ and
  \begin{align} \label{eq:var-vol}
     \Var(V_d(F_n))=\begin{cases}
        \frac{1-p}{M^dp-1} \left(p^{2n}-\left(\frac{p}{M^d}\right)^n\right),& p\neq M^{-d},\\
        (1-p)\cdot n\cdot p^{2n},& p= M^{-d}.
     \end{cases}
  \end{align}
\end{thm}

Observe that, for any $n\in\N$, the variance in \eqref{eq:var-vol} is continuous in $p$ even at the value $p=M^{-d}$. Indeed, for any $p$ it can be written as
$$
\Var(V_d(F_n))=p^{2n}(1-p) \sum_{i=0}^{n-1} (M^dp)^{i-n}.
$$
From Theorem~\ref{thm:var-vol} we can deduce the following `\emph{rates of the convergence}' for expectation and variance of the rescaled volume $Z^d_n=M^{(d-D)n} V_d(F_n)$ to expectation and variance, respectively, of the limit $Z^d_\infty$.
\begin{cor} \label{cor:var-vol}
If $p>M^{-d}$ then, as $n\to\infty$, $\E Z^d_n =1 \to 1=\E Z^d_\infty$ and
  \begin{align*}
    \Var(Z^d_\infty) - \Var(Z^d_n) = \frac{1-p}{M^dp-1} (M^dp)^{-n}\to 0. 
  \end{align*}
  If $p\leq M^{-d}$, then $Z^d_n\to 0$ a.s., as $n\to\infty$, while $\E Z^d_n=1\to 1\neq 0=\E Z^d_n$ and $\Var(Z^d_n)\to\infty$.
\end{cor}
\begin{proof}
   Let first $p>M^{-d}$. Since $M^D=M^dp$, we infer from Theorem~\ref{thm:var-vol} for any $n\in\N$ that $\E Z_n^d=p^{-n}\E V_d(F_n)=1$ and
   $$
   \Var(Z^d_n)=\Var(p^{-n} V_d(F_n))=p^{-2n} \Var(V_d(F_n))=\frac{1-p}{M^dp-1} \left(1-\left(M^dp\right)^{-n}\right).
   $$
   Moreover, by Theorem~\ref{thm:Vk-limit-general}, we have $\osV_d(F)=1$ and thus we conclude from Corollary~\ref{cor:l-indep-fps} and \eqref{eq:var_W_infty}, that $\E Z^d_\infty=\osV_d(F) \E W_\infty=1$ and $\Var(Z^d_\infty)=\frac{1-p}{M^dp-1}$. From this the asserted convergence of expectation and variance are obvious.

   If $p\leq M^{-d}$, then $F=\emptyset$ almost surely. Moreover, it is clear from the definition of $F$ that $F(\omega)=\emptyset$ for some $\omega\in\Omega$, if and only if there is some $m=m(\omega)\in\N$ such that $F_m(\omega)=\emptyset$, and in this case one has $Z^d_n(\omega)\to 0$, as $n\to\infty$. Hence $Z_n^d\to 0$ a.s., as $n\to \infty$ as stated. Now the last assertions are obvious from \eqref{eq:var-vol}.
\end{proof}
In the terminology of numerical analysis, Corollary~\ref{cor:var-vol} says that the variances converge linearly with rate $(M^dp)^{-1}<1$ whenever $p>M^{-d}$. This corresponds to our observations in simulations of the variables $Z_n^d$ for different parameters $p$ and $M$.
 In Section~\ref{sec:var_vol}, we will provide two proofs of Theorem~\ref{thm:var-vol}. The first one uses a standard branching process argument based on the observation that the volume of $F_n$ in this model is directly related to the number of offspring in the $n$-th generation of the associated Galton-Watson process. The second one uses a recursion argument, which generalizes to other intrinsic volumes and is a warm-up for the more involved discussion of the surface area.

For the surface area of $F_n$, we obtain the following explicit expressions. Here we concentrate on the case $p>M^{-d}$, when the limit set $F$ is nontrivial. Recall that $V_{d-1}$ is, in fact, half the surface area.

\begin{thm} \label{thm:var-surf}
Let $F$ 
be a fractal percolation on $[0,1]^d$ with parameters $M\in\N_{\geq 2}$ and $p\in(0,1]$. 
  Then, for each $n\in\N$,
  \begin{align} \label{eq:exp-surf-Exp}
  \E V_{d-1}(F_n)&= \bar c_1 \cdot (Mp)^n \left(1+\frac{M-1}{1-p}\left(\frac pM\right)^{n+1}\right),
  \end{align}
  where $\bar c_1:=\frac {dM(1-p)}{M-p}=\osV_{d-1}(F)$. Moreover, if $p>M^{-d}$, then, as $n\to\infty$,
  \begin{align}
   \label{eq:exp-surf-Var}
        \Var(V_{d-1}(F_n))=&\bar c_2\cdot (Mp)^{2n} + 
  \begin{cases}
    O((Mp^3)^n),& \text{if } p^2>1/M^{d-1},\\
    O((Mp^3)^n \cdot n), & \text{if } p^2=1/M^{d-1},\\
    O(\left(\frac{p}{M^{d-2}}\right)^n), & \text{if } p^2<1/M^{d-1},
  \end{cases} 
  \end{align}
  where $\bar c_2:=\bar c_1^2\cdot \frac{1-p}{M^dp-1}=(\osV_{d-1}(F))^2 \Var(W_\infty)$.
 \end{thm}

\begin{cor} \label{cor:var-surf} Let $p>M^{-d}$. Then, for any 
$n\in\N$,
\begin{align*}
  \E Z^{d-1}_\infty - \E Z^{d-1}_n = \frac {dp(M-1)}{M-p} \left(\frac pM\right)^{n},
\end{align*}
and, as $n\to\infty$, 
  \begin{align*}
    \Var(Z^{d-1}_\infty) - \Var(Z^{d-1}_n) =
  \begin{cases}
    O(\left(\frac pM\right)^n),& \text{if } p^2>1/M^{d-1},\\
    O(\left(\frac pM\right)^n \cdot n), & \text{if } p^2=1/M^{d-1},\\
    O(\left(M^{d}p\right)^{-n}), & \text{if } p^2<1/M^{d-1}.
  \end{cases}
  \end{align*}
  In particular, $\Var(Z^{d-1}_n) \to \Var(Z^{d-1}_\infty)$, as $n\to\infty$.
\end{cor}
\begin{proof}
   Since $M^{d-1-D}=(Mp)^{-1}$, we infer from Theorem~\ref{thm:var-surf} that
   $$
   \E Z^{d-1}_n=(Mp)^{-n}\E V_{d-1}(F_n)=\bar c_1 
   \left(1+\frac{M-1}{1-p}\left(\frac pM\right)^{n+1}\right)
   $$
   for any $n\in\N$. As noted above, Theorem~\ref{thm:Vk-limit-general} and Corollary~\ref{cor:l-indep-fps} imply that $\E Z^{d-1}_n \to \E Z^{d-1}_\infty$, as $n\to\infty$, for any $p>M^{-d}$, and therefore we conclude $\E Z^{d-1}_\infty=\bar c_1$ (which can also be deduced directly from Thm.~\ref{thm:Vk-limit-general}). The stated formula for the difference $\E Z^{d-1}_\infty - \E Z^{d-1}_n$ is now obvious. Moreover, from \eqref{eq:exp-surf-Var} we infer that
   \begin{align*}
     \Var(Z^{d-1}_n)
   &=(Mp)^{-2n} \Var(V_{d-1}(F_n))
   =\bar c_2 + \begin{cases}
    O(\left(\frac pM\right)^n),& \text{if } p^2>1/M^{d-1},\\
    O(\left(\frac pM\right)^n \cdot n), & \text{if } p^2=1/M^{d-1},\\
    O(\left(M^{d}p\right)^{-n}), & \text{if } p^2<1/M^{d-1},
  \end{cases}
   \end{align*}
   which implies in particular that $\Var(Z^{d-1}_n)\to \bar c_2$, as $n\to\infty$. From Corollary~\ref{cor:l-indep-fps} we see that $\Var(Z^{d-1}_\infty)=(\osV_k(F))^2\cdot \Var(W_\infty)=\bar c_2$. Hence $\Var(Z^{d-1}_n)\to \Var(Z^{d-1}_\infty)$, as $n\to\infty$, and the stated order of convergence for the difference follows at once.
\end{proof}

\section{Random iterated function systems and a renewal theorem for branching random walks} \label{sec:2}

Fractal percolation in $\R^d$ 
can be viewed as a random self-similar set, i.e., as a random compact set generated by a 
random iterated function system (RIFS) $\sS$. 

The general definition of such RIFSs is as follows, cf.~\cite{Falconer86,MW86,Graf87}. Let $J\subset\R^d$ be a compact set such that $J=\overline{\intr(J)}$ and let $\Sim$ be some family of contracting similarity mappings on $J$. Then an RIFS $\sS$ on $J$ is defined to be a random subset of $\Sim$, which is almost surely finite and satisfies the uniform open set condition with respect to $\intr(J)$. That is, there is a random variable $\nu$ with values in $\N_0=\{0\}\cup\N$ and random elements $\Phi_i\in\Sim$, $i=1,\ldots,\nu$ such that $\sS=\{\Phi_1,\Phi_2,\ldots, \Phi_\nu\}$, if $\nu>0$, and $\sS:=\emptyset$, if $\nu=0$.
Moreover, $\sS$ is said to satisfy the \emph{uniform open set condition} (UOSC) with respect to $\intr(J)$, if
\begin{align} \label{eq:OSC}
  \bigcup_{i=1}^\nu \Phi_i(\intr(J))\subset \intr(J) \quad \text{ and } \quad \Phi_i(\intr(J))\cap \Phi_j(\intr(J))=\emptyset, \quad i\neq j,
\end{align}
holds with probability 1. Additionally, we assume throughout that $\nu$ satisfies
\begin{align*}
  0<\E\nu<\infty.
\end{align*}

Given an RIFS $\sS$, a random fractal set $F$ can be associated to it by constructing a Galton-Watson tree on the code space $\N^*:=\bigcup_{n=0}^\infty \N^n$ of all finite words with letters in $\N$. Here $\N^0:=\{\eps\}$, where $\eps$ is the \emph{empty word}. For any word $\sigma\in\N^n$, $|\sigma|:=n$ will denote its \emph{length}. Moreover, for $\sigma,\tau\in\N^*$, we write $\sigma\tau$ for the concatenation of $\sigma$ and $\tau$.
For each $\sigma\in\N^*$, let $\sS_\sigma$ be a copy of the RIFS $\sS$ generated in its own probability space $(\Omega_\sigma,\mathcal{A}_\sigma,\P_\sigma)$. Let $(\Omega,\mathcal{A},\P):=\bigotimes_{\sigma\in\N^*} (\Omega_\sigma,\mathcal{A}_\sigma,\P_\sigma)$ be the common probability space in which all these RIFS are independent.
Recall that $\sS_\sigma$ contains a random number $\nu_\sigma$ of maps. 
To distinguish them, let $I_\sigma\subseteq \N$ be a set of indices with cardinality $|I_\sigma|=\nu_\sigma$. It is convenient to denote the maps in $\sS_\sigma$ by $\Phi_{\sigma i}$, $i\in I_\sigma$. Note that $\nu_\sigma$ may be $0$ in which case $I_\sigma=\emptyset$. (In general, one could choose $I_\sigma=\{1,\ldots,\nu_\sigma\}$ here without loss of generality, but later in the case of fractal percolation it will be much more convenient to use different index sets.)
We build a random tree $\sT$ in $\N^*$ as follows:
set $\sT_0:=\{\eps\}$ and define, for $n\in\N_0$, $\sT_{n+1}:=\emptyset$, if $\sT_n=\emptyset$, and
$$
\sT_{n+1}:=\{\sigma i: \sigma\in\sT_n, i\in I_\sigma\},
$$
if $\sT_n\neq\emptyset$. Finally, let
$$
\sT_*:=\bigcup_{n=0}^\infty \sT_n
$$
be the vertex set of the tree $\sT$ and define the edge set by
$$
E(\sT):=\{(\sigma,\sigma i):\sigma\in\sT_*, i\in I_\sigma\}.
$$
$\sT$ can be interpreted as the population tree of a \emph{Galton-Watson process} in which $\sT_n$ represents the $n$-th generation and $\sigma i\in\sT_{n+1}$, $i\in I_\sigma$ are the descendants of an individual $\sigma\in\sT_n$. The distribution of $\nu$ is called the \emph{offspring distribution} of the process. For any finite word $\sigma\in\N^n$ and any $k\in\N$, $k\leq n$, write $\sigma|k$ for the word of the first $k$ letters of $\sigma$.
With this notation, the self-similar random set $F$ associated with the RIFS $\sS$ is defined by
\begin{align} \label{eq:F-def}
F:=\bigcap_{n=1}^\infty \bigcup_{\sigma\in\sT_n} J_\sigma,
\end{align}
where, 
\begin{align}\label{eq:J_sigma-def}
   J_\sigma:=\overline{\Phi}_\sigma(J):=\Phi_{\sigma|1}\circ \Phi_{\sigma|2}\circ\ldots\circ\Phi_{\sigma|n-1}\circ \Phi_\sigma(J)
\end{align}
are the \emph{cylinder sets} of the construction.


\paragraph{\bf Fractal percolation as a random self-similar set} Before we continue the general discussion of RIFS, let us briefly indicate how
fractal percolation $F$ on $[0,1]^d$ with parameters $M\in\N_{\geq 2}$ and $p\in(0,1]$ fits into this setting. Choose $J:=[0,1]^d$ as the basic set. Recall the basic cubes $J_1,\ldots,J_{M^d}$ of sidelength $1/M$ into which $J$ is subdivided in the first step of the construction of $F$. %
Let $\Sim':=\{\varphi_1,\ldots, \varphi_{M^d}\}$, where for $j=1,\ldots, M^d$, $\varphi_j$ is the similarity which maps $J$ to $J_j$ (rotation and reflection free, for simplicity and uniqueness).
Let $\sS$ be the random subset of $\Sim'$ such that each of the maps $\varphi_j$ is contained in $\sS$ with probability $p$ independently of all the other maps in $\Sim'$. It is obvious that $\sS$ satisfies the UOSC with respect to the interior of $J$. Indeed, $\sS$ is a random subset of $\Sim'$ and any subset of $\Sim'$ satisfies condition \eqref{eq:OSC} with respect to $\intr(J)$. Now fractal percolation with parameters $M\in\N_{\geq 2}$ and $p\in(0,1]$ is the random self-similar set $F$ (defined by \eqref{eq:F-def}) generated by this particular RIFS $\sS$. Note that $\sS$ contains at most $M^d$ maps. One can therefore reduce the code space to $\Sigma^*:=\bigcup_{n\in\N_0} \Sigma^n$, where $\Sigma:=\{1,2,\ldots,M^d\}$. If we choose the index set $I_\sigma$ to contain exactly those indices $i\in\Sigma$ for which $\varphi_i\in \sS_\sigma$ and if we set $\Phi_{\sigma i}:=\varphi_i$,  then the cylinder sets $J_\sigma$, $\sigma=\sigma_1\ldots\sigma_n\in \Sigma^n$, defined in \eqref{eq:J_sigma-def} are more conveniently given by
\begin{align}\label{eq:J_sigma-def2}
   J_\sigma:=\varphi_{\sigma_1}\circ \ldots\circ\phi_{\sigma_n}(J).
\end{align}
They correspond to the basic cubes of level $n$ used in the previous sections. Note also that for $n=1$ this is consistent with the above notation $J_j$, $j\in\Sigma$ for the first level cubes.
In the language of the tree and the associated sets considered above, the construction steps $F_{n}$, $n\in\N$ of the fractal percolation process are given by
$$
F_{n}=\bigcup_{\sigma\in\sT_n} J_\sigma.
$$
For each $j\in\Sigma$ and each $n\in\N$, the random set $F_n^j$ introduced just before  Theorem~\ref{thm:Vk-limit-general} is then given by
\begin{align} \label{eq:Fj_ndef}
F_n^j=\bigcup_{{\sigma\in\sT_n},{\sigma|1=j}} J_\sigma.
\end{align}


\paragraph{\bf A branching random walk associated with $F$.} Let us go back to general RIFS. To any random self-similar set $F$, a branching random walk $\{S_\sigma: \sigma\in\N^*\}$ can naturally be associated. It controls the size of the cylinder sets of $F$ (i.e., of the basic cubes in the case of fractal percolation), by keeping track of the contraction ratios applied during the construction. In general, it is defined recursively by setting $S_\varepsilon:=0$ and, for $n\in\N$ and any $\sigma=\sigma_1\ldots\sigma_n\in\sT_n$ by
$$
S_{\sigma}:=S_{\sigma|n-1}+\log r^{-1}_{\sigma},
$$
where $r_\sigma$ is the contraction ratio of the similarity $\Phi_\sigma$. It is convenient to set $S_{\sigma}:=\infty$ for $\sigma\in \N^*\setminus\sT$. Then $\{S_\sigma: \sigma\in\N^*\}$ is a branching random walk with positive step sizes.
In the case of fractal percolation $F$ in $\R^d$ with subdivision parameter $M\geq 2$, 
all contraction ratios (of all $\Phi_\sigma$, $\sigma\in\sT)$ equal $1/M$.
Therefore, we have in this case, for any $n\in\N$, 
$$
S_\sigma=\begin{cases}
   n\cdot \log M,  & \sigma\in\sT_n,\\
   +\infty, & \text{otherwise. }
\end{cases}
$$

Let $\xi$ be the random measure on $\R$ defined by
\begin{align}
  \label{eq:xi} \xi(B):=\sum_{\sigma\in\sT_1} \ind_B(S_\sigma)
\end{align}
for any Borel set $B\subset\R$, and let $\mu:=\E[\xi(\cdot)]$ be the intensity measure of $\xi$. The random measure $\xi$ is called \emph{lattice}, if $\mu$ 
is concentrated on $\lambda \Z$ for some $\lambda>0$ (we also say $\xi$ is \emph{lattice with lattice constant} $\lambda$ in this case), and $\xi$ is called \emph{nonlattice} otherwise. Observe that in case of fractal percolation with subdivision parameter $M$, the intensity measure $\mu$ is concentrated on the value $\log M$ and so $\xi$ is clearly lattice with lattice constant $\lambda=\log M$.

For a general RIFS $\sS$, recall that $\nu=\#\sS=\#\sT_1$. Assume from now on that $$
1<\E\nu<\infty.$$ Denote by $D\in\R$ the number determined uniquely by the equation
$$
\E\left(\sum_{i=1}^\nu r_i^D\right)=1.
$$
Here $r_i$ denotes the contraction ratio of the mapping $\Phi_i$ in $\sS$.
It is well known that conditioned on $F\neq\emptyset$ the Hausdorff and Minkowski dimension of $F$ equal $D$ almost surely, cf.\ \cite{Falconer86, Graf87,MW86, Patzschke97}. Note that in case of fractal percolation in $\R^d$ with parameters $M$ and $p$, the assumption $\E\nu>1$ equals the condition $p>M^{-d}$ and the above $D$ equals the number given in \eqref{eq:dimF}.
Furthermore, we set
$$
m(D):=\E\left(\sum_{i\in\sT_1} |\log r_i| r_i^D\right).
$$
To the branching random walk $\{S_\sigma: \sigma\in\N^*\}$ we can associate the following nonnegative martingale $(W_n)_{n\in\N_0}$, which in the case of fractal percolation specializes to the one defined after equation \eqref{eq:variance-N-n}:
\begin{align}
  W_n:= \sum_{\tau\in\sT_n} e^{-D S_\tau}, \quad n\in\N_0.
\end{align}
Indeed, for fractal percolation with parameters $M$ and $p$, we have $S_\sigma=n \log M$ for $\sigma\in\sT_n$, and $|\sT_n|=N_n$ and thus $W_n = \sum_{\tau\in\sT_n} M^{-D n} = (M^dp)^{-n}N_n = N_n/\E N_n$.

Also in general, $(W_n)_n$ is a nonnegative martingale (with respect to the filtration $(\hat{\mathcal{F}}_n)_n$ where $\hat{\mathcal{F}}_n$ is the product $\sigma$-algebra generated by the $\mathcal{F}_\sigma$ with $|\sigma|\leq n$) with $\E W_n=1$ for each $n\in\N$ and, by the martingale convergence theorem, the almost sure limit
\begin{align}
  W_\infty:=\lim_{n\to\infty} W_n
\end{align}
is well defined. Moreover, by Biggins' Theorem \cite{biggins77}, see also \cite[Thm.~3.3]{Gatzouras00}, $W_\infty$ is nontrivial, i.e., $\P(W_\infty=0)<1$, if and only if $\E[W_1 \log^+W_1]<\infty$. (Note that this condition is satisfied for fractal percolation for all parameters such that $p>M^{-d}$.)

In \cite{Gatzouras00}, a renewal theorem for branching processes $Z_t$, $t\in\R$ associated to the random walk $S=\{S_\sigma: \sigma\in\N^*\}$  has been formulated, which is based on and extends a result of Nerman \cite{Nerman81}.
First we recall the theorem from \cite{Gatzouras00}, see also \cite{Z11}. Then we will reformulate this statement in order to apply it to some discrete processes.

Recall that the underlying probability space is 
$(\Omega,{\mathcal{F}},\P) =\prod_{\sigma\in\N^*} (\Omega_\sigma,{\mathcal{F}}_\sigma,\P_\sigma)$, where the components $(\Omega_\sigma,{\mathcal{F}}_\sigma,\P_\sigma)$ are identical and each of them generates an RIFS which, loosely speaking, determines the offspring of $F_\sigma$, i.e.\ the sets $F_{\sigma i}$, $i=1,\ldots,\nu_\sigma$. The tree $\sT_*$ induces a tree structure to the elements of $\Omega$ and it is easy to see that any subtree rooted at some $\tau\in\sT_*$ and containing all words starting with $\tau$ is equivalent in distribution to the full tree $\sT_*$. Using only the components of $\omega\in\sT_*$ determined by this subtree, we can generate a random self-similar set $F^{(\tau)}$ which is equal in distribution to $F$.  
More formally, define for each $\tau\in\N^*$ the shift operator $\theta_\tau:\Omega\to\Omega$ by
\begin{align*}
  (\theta_\tau\omega)_\sigma:= \omega_{\tau\sigma}.
\end{align*}
Since the images $\theta_\tau(\omega)$ are again in $\Omega$, we can define for any random variable $X$ on $\Omega$ (taking values in some arbitrary space $E$) a whole family of random variables $\{X^{(\tau)}: \tau\in\N^*\}$ by $X^{(\tau)}(\omega):=X(\theta_\tau\omega)$, $\omega\in\Omega$. They have the property that $X^{(\tau)}\deq X$ for any $\tau\in\N^*$.
In particular, we define the random set $F^{(\tau)}$ by
\begin{align}
   \label{eq:F-tau}
   F^{(\tau)}(\omega):=F(\theta_\tau\omega), \qquad \omega\in\Omega.
\end{align}
It satisfies  $F^{(\tau)}\deq F$ for any $\tau\in\N^*$.
Similarly, we define for any stochastic process $Y:=\{Y_t: t\in\R\}$ on $(\Omega,{\mathcal{F}},\P)$ a family of i.i.d.\ copies of $Y$ by $Y_t^{(\tau)}(\omega):= Y_t(\theta_\tau\omega)$, $\omega\in\Omega, \tau\in\N^*$. Then the \emph{branching process} associated with $S$ and $Y$ is defined by
\begin{align}
  Z_t:=\sum_{\sigma\in\sT_*} Y^{(\sigma)}_{t-S_\sigma}.
\end{align}
Now we are ready to recall the relevant part of the Nerman-Gatzouras renewal theorem.

\begin{thm}\cite[Thm 3.4, lattice case]{Gatzouras00}\label{thm:nerman-rt}\\
Let $\{Y_t: t\in\R\}$ be a stochastic process on $(\Omega,{\mathcal{F}},\P)(=\prod_{\sigma\in\N^*} (\Omega_\sigma,{\mathcal{F}}_\sigma,\P_\sigma))$, which is
continuous a.e.\ with probability $1$
and takes values in the space of functions 
which vanish on $(-\infty,0)$. Assume there exists a non-increasing and integrable function $h:[0,\infty)\to(0,\infty)$, such that
\begin{align}
\label{eq:rt-cond}
  \E\left[\sup_{t\geq 0} \frac{e^{-Dt}|Y_t|}{h(t)}\right]<\infty.
\end{align}
Assume further that the random measure $\xi$ (defined in \eqref{eq:xi}) is lattice with lattice constant $\lambda>0$ and let $s\in[0,\lambda)$. Then, as $n\to\infty$, almost surely
\begin{align*}
  e^{-Dn\lambda} Z_{s+\lambda n}= e^{-Dn\lambda} \sum_{\sigma\in\sT_*} Y^{(\sigma)}_{s+n\lambda-S_\sigma}\longrightarrow \frac{\lambda W_\infty}{m(D)} \sum_{n=0}^\infty e^{-Dn\lambda}\E[Y_{n\lambda+s}].
\end{align*}
\end{thm}

There is also a corresponding statement for the nonlattice case, which we omit here as we will not use it. The following discrete version of the previous theorem will be the main tool in the proof of our main result Theorem~\ref{thm:as-conv-Vk}. We only formulate it for the special case that the random measure $\xi$ is concentrated on a single value $\lambda>0$, which is the only case we need here. The statement can easily be generalized to any lattice random self-similar set (but the formulas are not as neat). Note that the assumption means that all contraction ratios are the same and equal $1/\Lambda=e^{-\lambda}$ almost surely. It implies in particular that $D= \log \E[\nu]/\log \Lambda$ and $m(D)=\lambda$.

\begin{prop} \label{prop:RTdiscrete}
Let $\{Y_n: n\in\N\}$ be a (discrete-time) stochastic process on $(\Omega,{\mathcal{F}},\P)$. 
Assume the random measure $\xi$ is almost surely concentrated on some $\lambda>0$, and set $\Lambda:=e^\lambda$.
Suppose there exists a non-increasing and summable sequence $(h_n)_{n\in\N}$, such that
\begin{align} \label{eq:RTdiscrete-cond}
  \E\left[\sup_{n\in\N} \frac{\Lambda^{-Dn}|Y_n|}{h_n}\right]<\infty.
\end{align}

Then almost surely, as $n\to\infty$,
\begin{align} \label{eq:RTdiscrete-concl}
  \Lambda^{-Dn} \sum_{\sigma\in\sT_*} Y^{(\sigma)}_{n-|\sigma|}\to {W_\infty} \sum_{n=0}^\infty \Lambda^{-Dn}\E[Y_{n}].
\end{align}
\end{prop}
\begin{proof}
  With the intention to apply Theorem~\ref{thm:nerman-rt} we consider the process $\tilde{Y}$ defined by $\tilde{Y}_t:=Y_n$, for $t\in I_n:=[n\log \Lambda, (n+1)\log \Lambda)$, $n\in\N_0$ and $\tilde{Y}_t:=0$, $t< 0$. Note that, for any $t\in I_n$,
  $$
  e^{-Dt} |\tilde{Y}_t|\leq e^{-Dn\log \Lambda} |Y_n|=\Lambda^{-Dn} |Y_n|.
  $$
Define the function $h:[0,\infty)\to(0,\infty)$ by 
$h(t):=h_n$ for $t\in I_n$, $n\in\N$. The assumed properties of the sequence $(h_n)$ imply that $h$ is non-increasing and integrable. Moreover, we have, for any $t\geq 0$,
\begin{align*}
  \frac{e^{-Dt}|\tilde{Y}_t|}{h(t)}\leq \frac{\Lambda^{-Dn}|{Y}_n|}{h_n}
\end{align*}
and thus (for any realization $\omega\in\Omega$)
\begin{align*}
  \sup_{t\geq 0}\frac{e^{-Dt}|\tilde{Y}_t|}{h(t)}=\sup_{n\in\N_0}\sup_{t\in I_n}\frac{e^{-Dt}|\tilde{Y}_t|}{h(t)}\leq \sup_{n\in\N_0}\frac{\Lambda^{-Dn}|{Y}_n|}{h_n}.
\end{align*}
Since, by assumption, the random variable on the right hand side has a finite expectation, so has the one on the left hand side.
Thus, the process $\tilde{Y}$ satisfies condition \eqref{eq:rt-cond} in Theorem~\ref{thm:nerman-rt} and we can apply this theorem to $\tilde{Y}$. (Note also that $\tilde Y$ vanishes on $(-\infty,0)$ and is continuous almost everywhere with probability 1.) Recalling that $\lambda=\log \Lambda$, we conclude in particular that (for $s=0$) the rescaled sum
$$
e^{-Dn\lambda} \sum_{\sigma\in\sT_*} \tilde{Y}^{(\sigma)}_{n\lambda-S_\sigma}=\Lambda^{-Dn} \sum_{\sigma\in\sT_*} \tilde{Y}^{(\sigma)}_{(n-|\sigma|)\log \Lambda}=\Lambda^{-Dn} \sum_{\sigma\in\sT_*} Y^{(\sigma)}_{n-|\sigma|}
$$
converges almost surely, as $n\to \infty$, to the random variable
\begin{align*}
\frac{\lambda W_\infty}{m(D)} \sum_{n=0}^\infty e^{-Dn\lambda}\E[\tilde{Y}_{n\lambda}]= W_\infty \sum_{n=0}^\infty \Lambda^{-Dn}\E[Y_{n}]. 
\end{align*}
\end{proof}

\section{Proof of Theorem~\ref{thm:as-conv-Vk}} \label{sec:proof}

We introduce some further notation. For $\tau\in \sT_*$, recall from \eqref{eq:F-tau} the definition of the random set $F^{(\tau)}$ and that it satisfies $F^{(\tau)}\deq F$. Furthermore, we will write $F^{[\tau]}:=\overline{\Phi}_\tau(F^{(\tau)})$ for the corresponding scaled copy of $F^{(\tau)}$ in $F$. 
We use $F^{(\tau)}_n$ for the $n$-th approximation of $F^{(\tau)}$, in particular $F^{(\tau)}_0=J$, and similarly, we let, for any $\tau\in\sT_*$ and $n\geq |\tau|$, $F^{[\tau]}_n:=\overline{\Phi}_\tau(F^{(\tau)}_{n-|\tau|})$. To see the relation with the notation $F^j_n$ introduced in \eqref{eq:Fj_ndef}, let  $F^\tau_n:=\bigcup_{\sigma\in\sT_n, \sigma||\tau|=\tau} J_\sigma$ for $n\geq|\tau|$. Then $F^{\tau}_n=F^{[\tau]}_n\cap \tilde J_\tau$, where $\tilde J_\tau$ is the random set which equals the cube $J_\tau$ provided $\tau\in\sT_*$ and is empty otherwise. Note that in particular
\begin{align}
  \label{eq:two-ways} F_n=\bigcup_{j=1}^{M^d} F_n^j=\bigcup_{j\in\sT_1} F_n^{[j]}.
\end{align}
Now we are ready to provide a proof of Theorem~\ref{thm:as-conv-Vk}.
\begin{proof}[Proof of Theorem~\ref{thm:as-conv-Vk}.] Fix $k\in\{0,\ldots,d\}$.
In order to express the functionals $Z^k_n$ in the form of the left hand side of \eqref{eq:RTdiscrete-concl}, we set $\widehat{Z}_n:=M^{nk} V_k(F_n)$. We will also write $\widehat{Z}^{(\sigma)}_n:=M^{nk} V_k(F^{(\sigma)}_n)$ for the corresponding functionals of the shifted random set $F^{(\sigma)}$, $\sigma\in\sT_*$. (To ease the notation, we suppress the dependence on $k$ here.) Observe that, by the inclusion-exclusion principle, we can decompose the random variables $\widehat{Z}_n$ for any $n\in\N$ as follows
\begin{align*}
  \widehat{Z}_n&=M^{nk} V_k(F_n)=M^{nk} V_k(\bigcup_{j\in\sT_1} {F}_n^{[j]})\\
 &= M^{nk}\sum_{j\in\sT_1} V_k(F_n^{[j]})+M^{nk}\sum_{T\subset{\sT}_1, |T|\geq 2} (-1)^{|T|-1} V_k(\bigcap_{j\in T} F_n^{[j]}).
\end{align*}
For any $n\in\N$, we define $Y_n$ to be the second of the two summands above and we set $Y_0:=0$ (which is consistent). Using the relation $V_k(F_n^{[j]})=M^{-k}V_k(F^{(j)}_{n-1})$ (recall that $F^{[j]}_n=\phi_j(F^{(j)}_{n-1})$, where $\phi_j$ has contraction ratio $1/M$ and that $V_k$ is homogeneous of order $k$), we infer that
\begin{align*}
   \widehat{Z}_n&=\sum_{j\in\sT_1} M^{(n-1)k} V_k(F^{(j)}_{n-1})+ Y_n=\sum_{j\in\sT_1} \widehat{Z}^{(j)}_{n-1}+ Y_n.
\end{align*}
Now each of the variables $\widehat{Z}^{(j)}_{n-1}$, $j\in\sT_1$, can be decomposed in the very same manner, which yields $\widehat{Z}^{(j)}_{n-1}=\sum_{\sigma\in\sT_2, \sigma|1=j} \widehat{Z}^{(\sigma)}_{n-2} + Y_{n-1}^{(j)}$ for $n\in\N_{\geq 2}$. For $n=1$ we simply get $\widehat{Z}^{(j)}_{n-1}=V_k(F^{(j)}_0)=V_k(J)=q_{d,k}$ for any $j\in\sT_1$. Iterating this procedure, we end up with one term $Y^{(\sigma)}_\ell$ for each finite word $\sigma\in\sT_*$ of length $|\sigma|$ up to $n-1$, i.e.\
we obtain
\begin{align} \label{eq:hatZ_n}
   \widehat{Z}_n&=\sum_{\sigma\in\sT_n} \widehat{Z}^{(\sigma)}_0+\sum_{\sigma\in\sT_*, |\sigma|<n} Y^{(\sigma)}_{n-|\sigma|}.
\end{align}
Here the first summand simplifies to $q_{d,k}|\sT_n|=q_{d,k} N_n$, since, for each $\sigma\in\sT_n$, $F^{(\sigma)}_0$ equals the unit cube $J$ and so $\widehat{Z}^{(\sigma)}_0= V_k(J)=q_{d,k}$.

We wish to study the limit of $Z^k_n=M^{-Dn} \widehat{Z}_n$, as $n\to\infty$, which we can do by studying separately the limits of the two sequences given by the two summands in \eqref{eq:hatZ_n}. The limiting behaviour of the first resulting sequence $(M^{-Dn} q_{d,k}N_n)$ is obvious.
Since $M^{Dn}=(M^{d}p)^{n}=\E N_n$, we conclude that
\begin{align}
  \label{eq:first-part} M^{-Dn} q_{d,k}N_n=q_{d,k} W_n \to q_{d,k} W_\infty \quad \text{ a.s., as } n\to\infty.
\end{align}
Thus the first summand shows the desired factorization with the first factor being deterministic and the second one being $W_\infty$.
The second sequence is of a form which allows to employ Proposition~\ref{prop:RTdiscrete}. In order to do so we need to verify that our  process $(Y_n)_n$ satisfies the assumptions of this statement. More precisely, we need to verify that
there is a non-increasing and summable sequence $(h_n)_{n\in\N}$, such that
\begin{align} \label{eq:RTdiscrete-cond2}
  \E\left[\sup_{n\in\N} \frac{M^{-Dn}|Y_n|}{h_n}\right]<\infty.
\end{align}
This can indeed be derived from the following Proposition~\ref{prop:main-estimate}, whose proof is postponed to the end of the section.
\begin{prop}
  \label{prop:main-estimate}
  Let $F$ be a fractal percolation in $[0,1]^d$ with parameters $M\in\N_{\geq 2}$ and $p\in(0,1]$. Then there exists a non-increasing and summable sequence $(h_n)$ such that, for any subset $T\subset\{1,2,\ldots, M^d\}$ with $|T|\geq 2$,
  \begin{align} \label{eq:RTdiscrete-cond3}
  \E\left[\sup_{n\in\N} \frac{M^{(k-D)n}}{h_n}|V_k(\bigcap_{j\in T} F^{[j]}_n)|\right]<\infty.
\end{align}
\end{prop}

Let us now show that this statement implies condition \eqref{eq:RTdiscrete-cond2}. First observe that, since $\sT_1$ is a random subset of $\Sigma=\{1,\ldots,M^d\}$, we have for any $n\in\N$
\begin{align*}
  |Y_n|&=M^{nk}\left|\sum_{T\subset{\sT}_1, |T|\geq 2} (-1)^{|T|-1} V_k(\bigcap_{j\in T} F_n^{[j]})\right|\\
  &\leq M^{nk}\sum_{T\subset{\sT}_1, |T|\geq 2}  \Big|V_k(\bigcap_{j\in T} F_n^{[j]})\Big|
   \leq\sum_{T\subset \Sigma, |T|\geq 2} M^{nk} \Big|V_k(\bigcap_{j\in T} F_n^{[j]})\Big|.
\end{align*}
Therefore, Proposition~\ref{prop:main-estimate} implies that there is some non-increasing and summable sequence $(h_n)$ such that
   \begin{align*}
  \E\left[\sup_{n\in\N} \frac{M^{-Dn}}{h_n}|Y_n|\right]
  &\leq \E\left[\sup_{n\in\N} \frac{M^{-Dn}}{h_n} \sum_{T\subset \Sigma, |T|\geq 2} M^{nk} \Big|V_k(\bigcap_{j\in T} F_n^{[j]})\Big|\right]\\
  &\leq \E\left[\sum_{T\subset \Sigma, |T|\geq 2} \sup_{n\in\N} \frac{M^{(k-D)n}}{h_n} \Big|V_k(\bigcap_{j\in T} F_n^{[j]})\Big|\right]\\
  &\leq \sum_{T\subset \Sigma, |T|\geq 2} \E\left[\sup_{n\in\N} \frac{M^{(k-D)n}}{h_n} \Big|V_k(\bigcap_{j\in T} F_n^{[j]})\Big|\right]<\infty,
\end{align*}
verifying condition \eqref{eq:RTdiscrete-cond2}. Hence we can apply Proposition~\ref{prop:RTdiscrete} to the sequence of the second summands in \eqref{eq:hatZ_n}. We infer that, as $n\to\infty$ ,
\begin{align}
  M^{-Dn} \sum_{\sigma\in\sT_*} Y^{(\sigma)}_{n-|\sigma|}\to {W_\infty} \cdot \sum_{n=1}^\infty M^{-Dn}\E[Y_{n}],
\end{align}
where, for any $n\in\N$,
\begin{align*}
  \E[Y_{n}]&=M^{kn}\E\Big[\sum_{T\subset{\sT}_1, |T|\geq 2} (-1)^{|T|-1} V_k(\bigcap_{j\in T} F_n^{[j]})\Big]\\
  &=M^{kn}\E\Big[\sum_{T\subset\Sigma, |T|\geq 2} (-1)^{|T|-1} V_k(\bigcap_{j\in T} F_n^{j})\Big].
\end{align*}
Here the last equality is due to \eqref{eq:two-ways}. Using this last representation, we get
\begin{align*}
  \sum_{n=1}^\infty M^{-Dn}\E[Y_{n}]
  &=\sum_{n=1}^\infty M^{(k-D)n}\sum_{T\subset\Sigma, |T|\geq 2} (-1)^{|T|-1} \E\Big[V_k(\bigcap_{j\in T} F_n^{j})\Big]\\
  &= \sum_{T\subset\Sigma, |T|\geq 2} (-1)^{|T|-1} \sum_{n=1}^\infty M^{(k-D)n} \E\Big[V_k(\bigcap_{j\in T} F_n^{j})\Big],
\end{align*}
where the interchange of the summations in the last equality is justified as long as all the (finitely many) series in the last expression converge. But this convergence has been shown in \cite[Proposition 5.1]{KW18} for any $p\in(0,1]$.
Recall that this is an expression for the limit of the second summands in \eqref{eq:hatZ_n}. Combining it with the limit of the first summands in \eqref{eq:first-part}, we see by comparison with \eqref{eq:Vk-limit-general}, that
\begin{align*}
 Z^k_\infty= \lim_{n\to\infty} Z_n^k= \osV_k(F)\cdot W_\infty \text{ a.s.},
\end{align*}
as asserted in Theorem~\ref{thm:as-conv-Vk}. This completes the proof.
\end{proof}

\begin{proof}[Proof of Proposition~\ref{prop:main-estimate}]
First observe that, for any index set $T\subset\Sigma$, the set $C:=\bigcap_{j\in T} J_j$  is a cube with sidelength $1/M$ and some dimension $u\leq d-1$. If $u=0$ the estimate in \eqref{eq:RTdiscrete-cond3} is obviously satisfied. So assume $u\geq 1$. For any $n\in\N$, the set $\bigcap_{j\in T} F_n^{[j]}$ is a subset of $C$ and thus at most $u$-dimensional. Note that in each of the sets $F_n^{[j]}$ only those level-$n$ cubes are relevant for the intersection $\bigcap_{j\in T} F_n^{[j]}$ which intersect $C$. This means we can model the independent random sets $F^{[j]}$,$j\in T$ as well by independent $u$-dimensional fractal percolations $K^{\{j\}}$, $j\in T$, constructed on the cube $C$ (with the same parameters $p$ and $M$ as $F$). That is, we have $F^{[j]}\cap C\deq K^{\{j\}}$ for each $j\in T$ and therefore, $F_n^{[j]}\cap C\deq K_n^{\{j\}}$ for each $n\in\N$. In particular, this implies
$\bigcap_{j\in T} F_n^{[j]}\deq \bigcap_{j\in T} K_n^{\{j\}}$. Choose some index $j_1$ from $T$ and denote by $N^u_n$ the number of level $n$ cubes contained in $K_n^{\{j_1\}}$. Recall that the sequence $(N^u_n)$ forms a Galton-Watson process with binomial offspring distribution $\text{Bin}(M^u, p)$.  Now we can apply \cite[Lemma 7.2]{KW18} according to which there is some constant $c_{u,k}$ (independent of $n$) such that
\begin{align*}
  |V_k(\bigcap_{j\in T} F_n^{[j]})|\deq
  |V_k(\bigcap_{j\in T} K_n^{\{j\}})|\leq C_k^{\var}(\bigcap_{j\in T}  K_n^{\{j\}}) \leq c_{u,k} M^{-kn} N^u_n. 
\end{align*}
Note that we deal here with unions of level $n$ subcubes of $C$ and not of the unit cube. However, this does only change the constant of the lemma by a factor $M^{-k}$, since we can scale the whole situation by factor $M$ to take place in a unit size cube. 
Multiplying by $M^{(k-D)n}$ and setting $W^{u}_n:=N^{u}_n/\E N^{u}_n$, where $\E N^{u}_n=(\E N^{u}_1)^n= M^{un}p^n$, we get
\begin{align*}
 M^{(k-D)n} |V_k(\bigcap_{j\in T} F_n^{[j]})|\leq c_{u,k} M^{(u-D)n}p^n W^{u}_n= c_{u,k} M^{(u-d)n} W^{u}_n. 
\end{align*}
Choosing now e.g.\ $h_n:=M^{-n/2}$, $n\in\N$ (which obviously forms a non-increasing and summable sequence), we conclude from the above inequality
that 
\begin{align}\label{eq:final}
\E\left[\sup_{n\in\N} \frac{M^{(k-D)n}}{h_n} |V_k(\bigcap_{j\in T} F_n^{[j]})|\right]\leq c_{u,k} \E\left[\sup_{n\in\N} M^{-\alpha n} W^{u}_n\right], 
\end{align}
where $\alpha:=d-u-\frac 12$. Note that $\alpha\geq\frac 12$, since $u\leq d-1$.
Now observe that the sequence of random variables  $X_m:=\sup_{n\in\{1,\ldots,m\}} M^{-\alpha n} W^{u}_n$, $m\in\N$, is non-decreasing and converges as $m\to\infty$ to $\sup_{n\in\N} M^{-\alpha n} W^{u}_n$. Moreover, we have for any $m\in\N$,
\begin{align*}
  \E\left[X_m\right]&\leq 
   \E\left[\sum_{\ell=1}^n M^{-\alpha \ell} W^{u}_\ell \right]=\sum_{\ell=1}^n M^{-\alpha \ell} \E\left[W^{u}_\ell\right]=\frac{1-M^{-\alpha n}}{M^{\alpha}-1}\leq \frac{1}{M^{\alpha}-1},
\end{align*}
which, by monotone convergence, implies that the right hand side of \eqref{eq:final} is bounded (by $c_{u,k}$ times the latter constant, in which $\alpha\geq 1/2$).
Finally, note that the chosen sequence $(h_n)$ is independent of the set $T$, and hence it can be used for all sets $T\subseteq\Sigma$, $|T|\geq 2$. Moreover, there are only finitely many choices for the constant $c_{u,k}$ (which depends on $T$ via the dimension $u$ of the resulting cube $C$). This shows the finiteness of the expectation in \eqref{eq:RTdiscrete-cond3} and completes the proof.
\end{proof}





\section[Variance of the volume of Fn]{Variance of the volume of $F_n$} \label{sec:var_vol}

Theorem~\ref{thm:var-vol} provides exact expressions for expectation and variance of the volume of $F_n$. We discuss two proofs of this result. A standard argument provides a short and elegant proof of the particular case of the volume, which does not generalize to other intrinsic volumes. The second one uses a recursion argument and avoids branching process techniques. It is ultimately based on the inclusion-exclusion principle and provides the idea how to prove analogous results for the other intrinsic volumes.

\begin{proof}[First proof of Theorem~\ref{thm:var-vol}.]
Observe that the volume $V_d(F_n)$ equals the number $N_n$ of cubes in the union $F_n$ times the volume of a single cube of level $n$, i.e.\ $V_d(F_n)=M^{-dn} N_n$, $n\in\N_0$. Recalling expectation and variance of $N_n$ from \eqref{eq:variance-N-n}, this implies
$\E(V_d(F_n))= M^{-dn} \E(N_n) = p^n$ and
\begin{align*}
    \Var(V_d(F_n))&=M^{-2dn} \Var(N_n)=\begin{cases}
     \frac{1-p}{M^dp-1} (M^{-d}p)^n\left((M^dp)^n-1\right),& \text{ if } p\neq 1/M^d,\\
     M^{-2dn}\cdot n\cdot (1-p), & \text { if } p=1/M^d,
  \end{cases}
\end{align*}
 from which formula \eqref{eq:var-vol} follows at once.
%
%
 \end{proof}

For the other intrinsic volumes $V_k(F_n)$, $k=0,\ldots,d-1$ a similar argument will not work since the mutual intersections of the cubes are essential for determining the functionals. Therefore we want to discuss a different proof of Theorem~\ref{thm:var-vol}, which provides an idea of how to proceed in the general case. It is based on simple recursions. We start by introducing some additional notation and some useful observation needed for the proof.

 Recall the definition of the code space $\Sigma^*$ and the encoding of the basic cubes $J_\sigma$ from \eqref{eq:J_sigma-def2}. Observe that $J_{\sigma \omega}\subset J_\sigma$ for all $\sigma,\omega\in\Sigma^*$. It will be convenient to consider all basic cubes and not only those whose ancestors survived previous steps of the construction. The construction steps $F_n$ can also be characterized as follows.

 For $n\in\N$ and $\sigma\in\Sigma_n$, let $Y_\sigma$ be the 0-1 random variable, which models the decision whether the subcube  $J_\sigma$ is kept in the $n$-th step of the construction of $F$. Observe that $\{Y_\sigma:\sigma\in\Sigma^*\}$ is a family of i.i.d.\ Bernoulli variables with parameter $p$. For $n\in\N$, we define the \emph{$n$-th layer} $L_n$ to be the random set given by
 \begin{align*}
   L_n:=\bigcup_{\sigma\in\Sigma^n, Y_\sigma=1} J_\sigma,
 \end{align*}
 and the \emph{$n$-th construction step} $F_n$ is then given by
 \begin{align*}
   F_n=\bigcap_{m=1}^n L_m.
 \end{align*}

For convenience set $F_0:=L_0:=J$, which is a deterministic set. The sets 
$F^j_{n}$ (defined in \eqref{eq:Fj_ndef}) can also be characterized by
$F^j_n=\bigcap_{m=1}^n L_m^j$, where
$$
L_m^j:=\bigcup_{\sigma\in\Sigma^m, \sigma|1=j, Y_\sigma=1} J_\sigma.
$$
It will be convenient to have also a notation for the resulting set in the $n$-th step if we ignore the first step of the construction. To this end, we define for any pair of indices $\ell,n\in\N$, with $\ell\leq n$ the sets
 \begin{align*}
   F_{\ell,n}:=\bigcap_{m=\ell}^n L_m \qquad \text{ and } \qquad  F_{\ell,n}^j:=\bigcap_{m=\ell}^n L_m^j, \quad j=1,\ldots, M^d.
 \end{align*}
 Note that $F_{1,n}=F_n$ and $F_{n,n}=L_n$ (and, similiarly, $F_{1,n}^j=F_n^j$ and $F_{n,n}^1=L_n^j$). The most important observation is that $F_{2,n}^j$ is a scaled version of $F_{n-1}$, that is, we have, for any $j\in\Sigma$,
 \begin{align} \label{eq:scaling-of-F2nj}
   F_{2,n}^j\deq\varphi_j(F_{n-1}).
 \end{align}

  When we study the intrinsic volumes of the $F_n$, it will convenient to separate the effect of the first step from the effect of the later steps.
  \begin{lem} \label{lem:self-sim-j}
     For any $n\in\N$, $k\in\{0,\ldots,d\}$ and $j\in\Sigma$,
   \begin{align*}
     V_k(F_n^j)\deq M^{-k} V_k(F_{n-1}) \cdot Y,
   \end{align*}
   where $Y$ is a Bernoulli variable with parameter $p$ independent of $F_{n-1}$. In particular, this implies that 
   \begin{align}
     \label{eq:self-sim-j-Exp}\E V_k(F_n^j)& =\frac p {M^{k}} \E V_k(F_{n-1}) \text{ and } \\
     \label{eq:self-sim-j-Var}\Var (V_k(F_n^j))&=\frac{p}{M^{2k}}\left(\Var( V_k(F_{n-1}))+ (1-p)(\E V_k(F_{n-1}))^2\right). 
   \end{align}
  \end{lem}
   \begin{proof}
      On the one hand, we have for any $n\in\N$, by \eqref{eq:scaling-of-F2nj} and the invariance properties of the intrinsic volumes,
  \begin{align} \label{eq:1-lem2-1}
    V_k(F_{2,n}^j)\deq V_k(\varphi_j(F_{n-1}))=M^{-k} V_k(F_{n-1}).
  \end{align}
 On the other hand,
  \begin{align}
    V_k(F_{n}^j)=V_k(L_1^j\cap F_{2,n}^j)=V_k(F_{2,n}^j)\cdot Y_j,
  \end{align}
  where $Y_j$ is the Bernoulli variable which determines whether $J_j$ is kept in the first step.
  Observe that $Y_j$ is independent of $F_{2,n}^j$. Combining both equations proves the first assertion. The expressions for expectation and variance follow from the first equation taking into account the independence of $Y$ and $F_{2,n}^j$ and recalling that $\E Y=\E Y^2=p$ and $\Var(Y)=p(1-p)$.
   \end{proof}

   Now we are ready to provide an alternative proof of Theorem~\ref{thm:var-vol} based on a direct recursion as announced.

  \begin{proof}[Second proof of 
  Theorem~\ref{thm:var-vol}]
  Recall that $F_n=\bigcup_{j=1}^{M^d} F_n^j$ and that from the perspective of volume this union is disjoint. Hence $V_d(F_n)=\sum_j V_d(F_n^j)$. Observe that this decomposes $V_d(F_n)$ into a sum of independent random variables.  According to Lemma~\ref{lem:self-sim-j}, the sets $F_n^j$ satisfy $\E V_d(F_n^j)= \frac p{M^{d}} \E V_d(F_{n-1})$, which implies that
  $$\E V_d(F_n)=\sum_j p M^{-d} \E V_d(F_{n-1})= p \cdot\E V_d(F_{n-1}) $$
  for any $n\in\N$, where $\E V_d(F_0)=1$. This is a recursion relation which implies immediately that $\E V_d(F_n)=p^n$. Furthermore, Lemma~\ref{lem:self-sim-j} yields also a recursion equation for the variance of $V_d(F_n)$
  \begin{align*}
    \Var(V_d(F_n))&=\Var\left(\sum_{j=1}^{M^d} V_d(F_n^j)\right)=\sum_{j=1}^{M^d}\Var\left(V_d(F_n^j)\right)\\
    &=\sum_{j=1}^{M^d} \frac{p}{M^{2d}}\left(\Var( V_d(F_{n-1}))+ (1-p)(\E V_d(F_{n-1}))^2\right)\\
    &= \frac{p}{M^{d}}\cdot \Var( V_d(F_{n-1})) + \frac{1-p}{M^{d}p}\cdot p^{2n},
  \end{align*}
  where we have used the independence of the $V_d(F_n^j)$ for the second equality. By induction, this leads to
  \begin{align*}
      \Var(V_d(F_n))&=\frac{p}{M^{d}} \Var(V_d(F_0))+\frac{1-p}{M^{d}p}\cdot \sum_{\ell=0}^{n-1} (\frac{p}{M^{d}})^{\ell} p^{2(n-\ell)},
  \end{align*}
  where the first summand vanishes, since $F_0=[0,1]^2$ is deterministic, and the second summand simplifies to the expression stated in Theorem~\ref{thm:var-vol}.
\end{proof}

\begin{rem}
  For the variances of the other intrinsic volumes similar (but more involved) recursions can be formulated. The decomposition $F_n=\bigcup_j F_n^j$ is still useful but, due to the inclusion-exclusion formula, it will lead to additional terms taking care of the intersection structure. Moreover, the functionals $V_k(F_n^j)$ are not independent anymore and their covariances have to be taken into account. In the last section we work this out for the surface area $V_{d-1}(F_n)$.
\end{rem}

The following statement is a generalization of Theorem~\ref{thm:var-vol} to the volume of the intersection of several fractal percolations generated independently on the same basic cube $J$. It will be used in the computations for the surface area, where some of the intersection terms can be interpreted as lower dimensional volume.
\begin{cor} \label{cor:l-indep-fps}
  Let $\ell\in\N$ and let $F^{\{1\}},\ldots, F^{\{\ell\}}$ 
be independent fractal percolations on $[0,1]^d$ with the same parameters $M\in\N_{\geq 2}$ and $p\in(0,1]$. 
  Then, for each $n\in\N$,
  $$
  V_d(F_n^{\{1\}}\cap\ldots\cap F_n^{\{\ell\}})\deq V_d(F'_n),
  $$
  where $F'$ is a fractal percolation on $[0,1]^d$ with parameters $M$ and $p^\ell$. In particular,
  \begin{align*} 
     \E V_d(F_n^{\{1\}}\cap\ldots\cap F_n^{\{\ell\}})&=p^{\ell n} \qquad \text{ and }\\
     \Var(V_d(F_n^{\{1\}}\cap\ldots\cap F_n^{\{\ell\}}))&=\begin{cases}
        \frac{1-p^\ell}{M^dp^\ell-1} \left(p^{2\ell n}-\left(\frac{p^\ell}{M^d}\right)^n\right),& p^\ell\neq M^{-d},\\
        (1-p^\ell)\cdot n\cdot p^{2\ell n},& p^\ell= M^{-d}.
     \end{cases}
  \end{align*}
\end{cor}
\begin{proof}
Let $N'_n$ denote the number of basic cubes of level $n$ contained in the intersection $I_n:=F_n^{\{1\}}\cap\ldots\cap F_n^{\{\ell\}}$. Note that a basic cube $J_\sigma$, $\sigma=\sigma_1\ldots\sigma_n\in\Sigma^n$, of level $n$ is contained in $I_n$ if and only if $J_{\sigma_1\ldots\sigma_{n-1}}$ is contained in $I_{n-1}$ and if for each $i=1,\ldots,\ell$, $Y_\sigma^{\{i\}}=1$, i.e.\ $J_\sigma$ survives in $F_n^{\{i\}}$. This implies that, for each basic cube of level $n-1$ contained in $I_{n-1}$, the random number of descendants (i.e.\ of subcubes of level $n$ contained in $I_n$) is $\text{Bin}(M^d, p^\ell)$-distributed, and hence $(N'_n)$ forms a Galton-Watson process that is equivalent to the one associated to a fractal percolation $F'$ on $[0,1]^d$ with parameters $M$ and $p^\ell$. Now the assertion follows from the fact, that
$V_d(I_n)$ equals $N'_n$ times $M^{-dn}$, the volume of a single basic cube of level $n$, which is also the volume $V_d(F'_n)$. Now the formulas for expectation and variance follow directly from Theorem~\ref{thm:var-vol}.
\end{proof}

\section[Expectation and variance of the surface area of Fn]{Expectation and variance of the surface area of $F_n$} \label{sec:var_surf}

For computing expectation and variance of the surface area of the $n$-th step $F_n$ of fractal percolation, one has to study those intersections $F_n^j\cap F_n^\ell$ for which $J_j$ and $J_\ell$ are direct neighbors. We call two basic cubes of level $n$ \emph{direct neighbors} if they are distinct and share a common facet. For the computations the following statement will be helpful. It will be combined below with Corollary~\ref{cor:l-indep-fps}.

\begin{lem} \label{lem:self-sim-j-2}
     For any $n\in\N$, $k\in\{0,\ldots,d\}$ and $j,\ell\in\Sigma$ such that $J_j$ and $J_\ell$ are direct neighbors,
   \begin{align*}
     V_k(F_n^j\cap F_n^\ell)\deq M^{-k} V_k(K_{n-1}^{\{1\}}\cap K_{n-1}^{\{2\}}) \cdot Y,
   \end{align*}
   where $K^{\{1\}}$ and $K^{\{2\}}$ are independent fractal percolations in $[0,1]^{d-1}$ (with the same parameters $M$ and $p$ as $F$), $K_{n}^{\{1\}}$, $K_{n}^{\{2\}}$ are their $n$-th construction steps and $Y$ is a Bernoulli variable with parameter $p^2$ independent of $K^{\{1\}}$, $K^{\{2\}}$.\\ In particular, one obtains that 
   \begin{align*}
     \E V_k(F_n^j\cap F_n^\ell)& =\frac{p^2}{M^{k}} \E V_k(K_{n-1}^{\{1\}}\cap K_{n-1}^{\{2\}}) \text{ and } \\
     \Var (V_k(F_n^j\cap F_n^\ell))&=\frac{p^2}{M^{2k}}\left(\Var( V_k(K_{n-1}^{\{1\}}\cap K_{n-1}^{\{2\}}))\right.\\&\left.\qquad\qquad + (1-p^2)(\E V_k(K_{n-1}^{\{1\}}\cap K_{n-1}^{\{2\}}))^2\right). 
   \end{align*}
  \end{lem}
   \begin{proof}
First observe that, similarly as in the proof of Lemma~\ref{lem:self-sim-j},
  \begin{align}
    V_k(F_{n}^j\cap F_n^\ell)=V_k(L_1^j\cap F_{2,n}^j\cap L_1^\ell\cap F_{2,n}^\ell)=V_k(F_{2,n}^j\cap F_{2,n}^\ell)\cdot Y_j\cdot Y_\ell,
  \end{align}
  where $Y_j$ and $Y_\ell$ are the Bernoulli variables which determine whether $J_j$ and $J_\ell$, respectively, are kept in the first step.
  Note that $Y_j, Y_\ell, F_{2,n}^j, F_{2,n}^\ell$ are independent and that the product $Y_j\cdot Y_\ell$ is a Bernoulli variable with parameter $p^2$. Now observe that, by \eqref{eq:scaling-of-F2nj}, and with $K^{\{i\}}$ as in the statement, 
  for any $n\in\N$, $F_{2,n}^j\cap F_{2,n}^\ell\deq \varphi(K^{\{1\}}_{n-1}\cap K^{\{2\}}_{n-1})$, where $\varphi:\R^{d-1}\to \R^d$ is the similarity with contraction ratio $1/M$ mapping $[0,1]^{d-1}$ to $J_j\cap J_\ell$.
  This implies in particular that
  \begin{align}
    V_k(F_{2,n}^j\cap F_{2,n}^\ell)\deq V_k(\varphi(K^{\{1\}}_{n-1}\cap K^{\{2\}}_{n-1}))=M^{-k} V_k(K^{\{1\}}_{n-1}\cap K^{\{2\}}_{n-1}).
  \end{align}
  Combining both equations proves the first assertion. The expressions for expectation and variance follow easily.
   \end{proof}


Now we have all the ingredients to start the proof of Theorem~\ref{thm:var-surf}.


\begin{proof}[Proof of Theorem~\ref{thm:var-surf}] Fix $n\in\N$. For $j,\ell\in\Sigma$, 
let $X_n:=V_{d-1}(F_n)$, $X_n^{j}:=V_{d-1}(F_n^j)$ and
  \begin{align*}
    X_n^{j,\ell}:=V_{d-1}(F_n^j\cap F_n^\ell).
  \end{align*}
  Note that all these random variables are nonnegative. The variables $X_n^j$, $j\in\Sigma$ are independent and identically distributed. Note also that $X_n^{j,\ell}=0$, whenever the corresponding cubes $J_j$ and $J_\ell$ are not direct neighbors.
  By the inclusion-exclusion principle, we have
  \begin{align} \label{eq:incl-excl}
     X_n&=\sum_{j\in\Sigma} X_n^j - \sum_{\{j,\ell\}\subset\Sigma} X_n^{j,\ell}.
  \end{align}
  For the expectation this implies
  \begin{align*}
    \E(X_n)&=\sum_{j\in\Sigma} \E(X_n^j) - \sum_{\{i,\ell\}\subset\Sigma} \E\left(X_n^{i,\ell}\right).
  \end{align*}
  By Lemma~\ref{lem:self-sim-j}, the first sum simplifies to $Mp \E(X_{n-1})$. In the second term it is enough to sum over those couples $\{i,\ell\}$ for which the corresponding cubes $J_i$ and $J_\ell$ are direct neighbors. 
  There are $d (M-1) M^{d-1}$ such couples in dimension $d$. All other summands are zero. Now observe that to direct neighbors 
  Lemma~\ref{lem:self-sim-j-2} can be applied, which in combination with Corollary~\ref{cor:l-indep-fps} yields
  \begin{align} \label{eq:exp-Xnil}
     \E\left(X_n^{i,\ell}\right)=\frac{p^2}{M^{d-1}} \E V_{d-1}(K_{n-1}^{\{1\}}\cap K_{n-1}^{\{2\}})=\frac{p^2}{M^{d-1}} p^{2(n-1)}=\frac{p^{2n}}{M^{d-1}}.
  \end{align}
  Therefore, we conclude that, for any $n\in\N$,
  \begin{align*}
     \E(X_n)&=Mp \E(X_{n-1})-d (M-1) M^{d-1}\cdot M^{1-d} p^{2n}=Mp \E(X_{n-1})-d (M-1)p^{2n},
  \end{align*}
  where $\E(X_0)=V_{d-1}(J)=d$. This is a recursion relation for the expectation and leads to
   \begin{align*}
     \E(X_n)&=(Mp)^n \E(X_{0})-d (M-1) \sum_{i=1}^n (Mp)^{n-i} p^{2i}\\
     &=d (Mp)^n \left(1-\frac{(M-1)p}{M-p}\left[1-\left(\frac pM\right)^n\right]\right),
  \end{align*}
  which is equivalent to the expression stated in equation \eqref{eq:exp-surf-Exp}, completing the proof of this equation.

  For the computation of the variance of $X_n$, the following statement is useful.

  \begin{lem} \label{lem:cov-basics}
  For $i=1,2$, let $Z_i,Z_i'$ be stochastically independent random variables. Then
  \begin{align*}
    \Cov(Z_1\cdot Z_1', Z_2\cdot Z_2')=\Cov(Z_1,Z_2)\cdot \E(Z_1'\cdot Z_2')+\E Z_1\cdot \E Z_2\cdot \Cov(Z_1', Z_2').
   \end{align*}
    In particular,
    $\Var(Z_1\cdot Z_1')=\Var(Z_1)\cdot \E(Z_1'\cdot Z_1')+(\E Z_1)^2\cdot \Var(Z_1').$
\begin{proof}  A straightforward computation that uses the assumed independence yields the first formula,
and the second one is an easy consequence obtained by setting $Z_1=Z_2$ and $Z_1'=Z_2'$.
\end{proof}
\end{lem}

  Now we are ready to compute the variance of $X_n$. Our starting point is again the central equation \eqref{eq:incl-excl}, which allows to split the variance $\Var(X_n)$ into several terms which will then be treated separately. Equation \eqref{eq:incl-excl} implies that
  \begin{align}
    \notag \Var(X_n)&=\Var\left(\sum_{j\in\Sigma} X_n^j\right) +
    \Var\left( \sum_{\{i,\ell\}\subset\Sigma} X_n^{i,\ell}\right)-2\cdot\Cov\left(\sum_{j\in\Sigma} X_n^j,\sum_{\{i,\ell\}\subset\Sigma} X_n^{i,\ell}\right)\\
   \label{eq:sur-var-proof1} &= \sum_{j\in\Sigma}\Var(X_n^j) + \sum_{\{i,\ell\}\subset\Sigma} \Var(X_n^{i,\ell})+ \sum_{
    \{i,\ell\}\neq\{i',\ell'\}} \Cov(X_n^{i,\ell},X_n^{i',\ell'})\\
   \notag  &\qquad -2\sum_{j\in\Sigma}\sum_{\{i,\ell\}\subset\Sigma}\Cov\left(X_n^j,X_n^{i,\ell}\right),
  \end{align}
  where we have used the independence of the $X_n^j$ in the first sum. Note that the sum in the third term extends over all pairs $(\{i,\ell\},\{i',\ell'\})$ of subsets of $\Sigma$ with two elements such that $\{i,\ell\}\neq\{i',\ell'\}$. Since the $X_n^j$ are identically distributed, we get for the first term above
    \begin{align*}
  \sum_{j\in\Sigma}\Var(X_n^j)&=M^d\cdot \Var(X_n^1)=M^d \frac{p}{M^{2(d-1)}}\left(\Var( X_{n-1})+ (1-p)(\E X_{n-1})^2\right)\\
  &=\frac{p}{M^{d-2}}\Var( X_{n-1})+ \frac{p(1-p)}{M^{d-2}}(\E X_{n-1})^2,
  \end{align*}
where for the second equality we employed equation \eqref{eq:self-sim-j-Exp} of Lemma~\ref{lem:self-sim-j}. By equation \eqref{eq:exp-surf-Exp}, we have
\begin{align*}
  (\E X_{n-1})^2&=(Mp)^{2n} \frac {d^2(1-p)^2}{p^2(M-p)^2} \left(1+\frac{M-1}{1-p}\left(\frac pM\right)^{n}\right)^2
\end{align*}
and thus
   \begin{align} \label{eq:sur-var-recursion1}
  \sum_{j\in\Sigma}\Var(X_n^j)&=\frac{p}{M^{d-2}}\Var( X_{n-1})+ \frac {d^2(1-p)^3}{M^{d-2}p(M-p)^2}\times\\ \notag &\quad\times(Mp)^{2n} \left(1+ \frac{M-1}{1-p}\left(\frac pM\right)^n\right)^2.
  \end{align}
  Inserting this expression into equation \eqref{eq:sur-var-proof1} yields a recursion relation for the variance $\Var(X_n)$.  In order to be of use, it remains to find explicit expressions for the other terms in \eqref{eq:sur-var-proof1}. The discussion below of the three remaining terms will show that all of them are of some order lower than $(Mp)^{2n}$ and thus do not contribute to first term in \eqref{eq:exp-surf-Var}. Moreover, it will become clear that for certain parameter combinations the last term in \eqref{eq:sur-var-proof1} provides a contribution to the second term in \eqref{eq:exp-surf-Var}.

  \paragraph{ \bf Analysis of the 2nd term in \eqref{eq:sur-var-proof1}} First recall that $X_n^{j,\ell}=V_{d-1}(F_n^j\cap F_n^\ell)=0$, whenever the corresponding cubes $J_j$ and $J_\ell$ are not direct neighbors, which implies  $\Var(X_n^{j,\ell})=0$ in this case. Recall also that there are $d(M-1)M^{d-1}$ pairs $\{j,\ell\}$ such that $J_j$ and $J_\ell$ are direct neighbors. By Lemma~\ref{lem:self-sim-j-2}, we get for all such pairs $\{j,\ell\}$ that $\Var(X_n^{j,\ell})$ equals 
  \begin{align*}
     &\frac{p^2}{M^{2(d-1)}}\left(\Var( V_{d-1}(K_{n-1}^{\{1\}}\cap K_{n-1}^{\{2\}}))+ (1-p^2)(\E V_{d-1}(K_{n-1}^{\{1\}}\cap K_{n-1}^{\{2\}}))^2\right),
  \end{align*}
where $K^{\{1\}}, K^{\{2\}}$ are independent $d-1$-dimensional fractal percolations. Hence, by Corollary~\ref{cor:l-indep-fps}, we have $(\E V_{d-1}(K_{n-1}^{\{1\}}\cap K_{n-1}^{\{2\}}))^2=p^{4(n-1)}$ and, for $p^2\neq M^{1-d}$,
\begin{align*}
  \Var( V_{d-1}(K_{n-1}^{\{1\}}\cap K_{n-1}^{\{2\}}))&=p^{4(n-1)} \frac{1-p^2}{M^{d-1}p^2-1}\left(1-\left(\frac 1{M^{d-1}p^2}\right)^{n-1}\right).
\end{align*}
For $p^2= M^{1-d}$, a different expression for the variance is provided by Corollary~\ref{cor:l-indep-fps}.
Inserting these expressions into the above formula yields
\begin{align}\label{eq:var-Xnjl-3}
     \Var(X_n^{j,\ell})&=\frac{1-p^2}{M^{d-1}}\cdot p^{4n} \begin{cases}
 \frac{1}{M^{d-1}p^2-1} \left(1 - \left(M^{d-1}p^{2}\right)^{-n}\right), &M^{d-1}p^2\neq 1,\\
  n, &M^{d-1}p^2= 1.
     \end{cases}
  \end{align}
  This provides explicit expressions for the 2nd term in \eqref{eq:sur-var-proof1}.

  \paragraph{\bf Analysis of the 3rd term in \eqref{eq:sur-var-proof1}} First note that the covariances of the form $\Cov(X_n^{i,\ell},X_n^{i',\ell'})$ appearing in the third term vanish whenever $\{i,\ell\}\cap\{i',\ell'\}=\emptyset$, simply because the variables $X_n^{i,\ell}$ and $X_n^{i',\ell'}$ are determined by disjoint subfamilies of the basic random variables $Y_\sigma$, $\sigma\in\Sigma^*$ in this case and are therefore independent. Since also $\{i,\ell\}\neq\{i',\ell'\}$, all relevant terms in the third sum refer to pairs of the form $(\{i,\ell\},\{i,\ell'\})$ with $\ell\neq\ell'$, where $J_\ell$ as well as $J_{\ell'}$ are direct neighbors of $J_i$. Now observe that the covariance $\Cov(X_n^{i,\ell},X_n^{i,\ell'})$ depends on the relative position of $J_\ell$ and $J_{\ell'}$, since this determines which of the basic variables $Y_\sigma$ determining $F_n^i$ are relevant for both quantities $X_n^{i,\ell}$ and $X_n^{i,\ell'}$. Since $J_i\cap J_\ell$ and $J_i\cap J_{\ell'}$ are facets of $J_i$ and any two facets of a cube are either disjoint (if they lie on opposite sides, i.e.\ in parallel hyperplanes) or intersect in a $(d-2)$-face (i.e. in a cube of dimension $d-2$), we are left with exactly these two cases. The remaining task is now to determine how many covariance terms there are for each of these two different types and to provide an explicit expression for each of them. We start with the easier one, for which the cubes $J_\ell$, $J_i$ and $J_{\ell'}$ have to be arranged in a row (in this order) in one of the $2d$ directions. (Note that due to the independent summation through all couples $\{i,\ell\}$ and all couples $\{i,\ell'\}$, it is a different contribution if $\ell$ and $\ell'$ are interchanged.) There are $2d(M-2)M^{d-1}$ such terms. (For $M=2$ this situation is not possible and accordingly this number is zero.)
  The only basic variable which affects both quantities $X_n^{i,\ell}$ and $X_n^{i,\ell'}$ is $Y_i$, i.e.\ only the first construction step affects this covariance term.   To compute it, we separate the influence of the first construction step from the rest (just as in Lemma~\ref{lem:self-sim-j}) and use the independence. We have for any $n\in\N$,
  \begin{align*}
     \Cov(X_n^{i,\ell},X_n^{i,\ell'})&= \Cov\left(V_{d-1}(F_{2,n}^i\cap L_1^i\cap F_n^\ell),V_{d-1}(F_{2,n}^i\cap L_1^{i}\cap F_n^{\ell'})\right)\\
     &=\Cov\left(V_{d-1}(F_{2,n}^i\cap F_n^\ell)\cdot Y_i,V_{d-1}(F_{2,n}^i\cap F_n^{\ell'})\cdot Y_i\right).
     \end{align*}
     Setting $\widetilde{X}_n:=V_{d-1}(F_{2,n}^i\cap F_n^\ell)$ and $\widetilde{X}_n':=V_{d-1}(F_{2,n}^i\cap F_n^{\ell'})$ and noting that these two variables are independent, identically distributed and independent of $Y_i$ with mean given by
     \begin{align} \label{eq:exp-Xn-Xnil}
       \E X_n^{i,\ell}= \E(V_{d-1}(F_{2,n}^i\cap F_n^\ell)\cdot Y_i)=\E \widetilde{X}_n\cdot \E Y_i=\E \widetilde{X}_n\cdot p,
     \end{align}
     we obtain, using the first formula in Lemma~\ref{lem:cov-basics}, that
  \begin{align*}
     \Cov(X_n^{i,\ell},X_n^{i,\ell'})&= \Cov\left(\widetilde{X}_n,\widetilde{X}_n'\right) \E(Y_i\cdot Y_i)+ \E \widetilde{X}_n\cdot \E \widetilde{X}_n'\cdot \Cov(Y_i,Y_i)\\
     &=\left(\E \widetilde{X}_n\right)^2 \Var(Y_i) =\frac {1-p}{p} \left(\E X_n^{i,\ell}\right)^2.
     \end{align*}
     By equation~\eqref{eq:exp-Xnil}, this yields for the $2d(M-2)M^{d-1}$ terms for which the corresponding cubes $J_\ell$, $J_i$ and $J_{\ell'}$ are arranged in a row, and for any $n\in\N$,
     \begin{align} \label{eq:3rd-term-in-a-row}
     \Cov(X_n^{i,\ell},X_n^{i,\ell'})&= \frac {1-p}{p} \left(\frac{p^{2n}}{M^{d-1}}\right)^2=\frac {1-p}{M^{2(d-1)}}p^{4n-1}.
     \end{align}
     Now let us look at the second type of covariance term occurring in the third sum in \eqref{eq:sur-var-proof1}, for which the involved cubes $J_i$, $J_\ell$ and $J_{\ell'}$ intersect in a $(d-2)$-face of $J_i$. First let us determine the number of such terms. It is convenient, to determine first the number of $(d-2)$-dimensional cubes, at which this intersection can happen. There are $\binom{d}{2}(M-1)^2M^{d-2}$ such cubes. (Indeed, choose $d-2$ out of $d$ directions to span the direction space $L$ of (the affine hull of) the cube (or equivalently, choose 2 out of $d$ to span the orthogonal complement of $L$. For each direction space there are $(M-1)^2M^{d-2}$ such cubes, $(M-1)^2$ in each $2$-dimensional ``layer'', and there are $M^{d-2}$ such layers.)  Each such $(d-2)$-cube $C$ is surrounded by four $d$-dimensional cubes, hence there are four choices for $J_i$. Then the two cubes which are direct neighbors of $J_i$ must be $J_\ell$ and $J_{\ell'}$, and the only choice we can make is which one is $J_\ell$. Hence for each $(d-2)$-cube there are eight possible choices and so the total number of terms of the second type in the third sum of \eqref{eq:sur-var-proof1} is $8\binom{d}{2}(M-1)^2M^{d-2}$. It remains to compute the covariance for this type.

     Setting (in analogy with the first type)
     $$
     \widehat{X}_n:=V_{d-1}(F_{2,n}^i\cap F_{2,n}^\ell) \text{ and } \widehat{X}_n':=V_{d-1}(F_{2,n}^i\cap F_{2,n}^{\ell'})
     $$
     and noting that these two variables are independent, identically distributed and independent of $Y_i, Y_\ell$, and $Y_{\ell'}$, we obtain for the mean
     \begin{align} \label{eq:exp-Xn-Xnil2}
       \E X_n^{i,\ell}= \E(V_{d-1}(F_{2,n}^i\cap F_{2,n}^\ell)\cdot Y_i\cdot Y_\ell)=\E \widehat{X}_n\cdot \E Y_i\cdot \E Y_\ell=\E \widehat{X}_n\cdot p^2,
     \end{align}
     and for the covariance, using the first formula in Lemma~\ref{lem:cov-basics} and \eqref{eq:exp-Xnil},
  \begin{align}
     \Cov(X_n^{i,\ell},X_n^{i,\ell'})&= \Cov\left(\widehat{X}_n\cdot Y_i Y_\ell,\widehat{X}_n'\cdot Y_iY_{\ell'}\right) \notag\\
     &=\Cov\left(\widehat{X}_n,\widehat{X}_n'\right) \E(Y_i^2 Y_\ell Y_{\ell'})+ \E \widehat{X}_n\cdot \E \widehat{X}_n'\cdot \Cov(Y_i Y_\ell,Y_i Y_{\ell'})\notag\\
     &=\Cov\left(\widehat{X}_n,\widehat{X}_n'\right) \cdot p^3+\left(\E \widehat{X}_n\right)^2 p^3(1-p) \notag\\
     &=\Cov\left(\widehat{X}_n,\widehat{X}_n'\right) \cdot p^3+\frac{1-p}{M^{2(d-1)}}\cdot p^{4n-1}. \label{eq:cov-Xnil}
      \end{align}
     But this time $\widehat{X}_n$ and $\widehat{X}_n'$ are not independent and so the covariance term does not vanish. In order to derive a recursion for the left hand side, notice that the intersection $J_i\cap J_\ell$ is $(d-1)$-dimensional and intersects $M^{d-1}$ of the second level cubes $J_{ij}$ contained in $J_i$. The intersection $J_i\cap J_\ell\cap J_{\ell'}$ is $(d-2)$-dimensional and intersects $M^{d-2}$ of the second level cubes $J_{ij}$ contained in $J_i$. Denote those latter cubes by $J_{i1}, J_{i2},\ldots, J_{iM^{d-2}}$ and let $R$ be the union of all the remaining second level cubes contained in $J_i$, i.e., set $R:=\cup\{J_{ij}:j\in\Sigma\setminus\{1,\ldots,M^{d-2}\}\}$. Observe that
     \begin{align*}
       \widehat{X}_n= V_{d-1}(F_{2,n}^i\cap F_{2,n}^\ell)=\underbrace{V_{d-1}(F_{2,n}^i\cap R\cap F_{2,n}^\ell)}_{=:\widetilde{R}}+\sum_{k=1}^{M^{d-2}} \underbrace{V_{d-1}(F_{2,n}^i\cap J_{ik}\cap F_{2,n}^\ell)}_{=:\widetilde{Q}_k},
     \end{align*}
     where all the terms on the right are independent. The same holds with $\ell$ replaced by $\ell'$, for which we denote the corresponding terms by $\widetilde{R}'$ and $\widetilde{Q}_k'$, respectively. What happens in $J_{ik}$ is now again a scaled version of what happens in $J_i$. More precisely, we have, for any $k\in\{1,\ldots, M^{d-2}\}$,
     \begin{align} \label{eq:3rd-term-type2}
        \widetilde{Q}_k & 
        \deq M^{-d+1}V_{d-1}(F_{n-1}^i\cap F_{n-1}^\ell)=M^{-d+1} X_{n-1}^{i,\ell}.
     \end{align}
     (Note that the corresponding sets are not equal in distribution due to intersections with neighboring cubes which are fortunately of lower dimension and thus do not contribute to the intrinsic volume of order $d-1$.)
     Taking into account the dependence structure of $\widetilde{R},\widetilde{R}',\widetilde{Q}_k,\widetilde{Q}_k$ and \eqref{eq:3rd-term-type2}, we conclude that
     \begin{align*}
        \Cov\left(\widehat{X}_n,\widehat{X}_n' \right)&=\Cov\left(\widetilde{R}+\sum_k \widetilde{Q}_k,\widetilde{R}'+\sum_k \widetilde{Q}_k'\right)\\
        &=\sum_k \Cov\left(\widetilde{Q}_k,\widetilde{Q}_k'\right)\\
        &= M^{d-2} \cdot \Cov\left(M^{1-d}\cdot X_{n-1}^{i,\ell},M^{1-d}\cdot X_{n-1}^{i,\ell'}\right)\\
        &= M^{-d} \Cov\left(X_{n-1}^{i,\ell},X_{n-1}^{i,\ell'}\right).
     \end{align*}
     Plugging this into \eqref{eq:cov-Xnil}, yields the desired recursion for the covariance of the second type: 
     \begin{align}
        \Cov(X_n^{i,\ell},X_n^{i,\ell'})&= \frac {p^3}{M^d}\cdot \Cov(X_{n-1}^{i,\ell},X_{n-1}^{i,\ell'}) + \frac{1-p}{M^{2(d-1)}}\cdot p^{4n-1},\qquad n\in\N,
     \end{align}
     where $\Cov(X_0^{i,\ell},X_0^{i,\ell'}):=0$. We conclude that, for $n\in\N$,
     \begin{align}
        \Cov(X_n^{i,\ell},X_n^{i,\ell'}) &=  \frac{1-p}{M^{2(d-1)}}\cdot p^{4n-1} \sum_{k=0}^{n-1} \left(\frac 1{M^dp}\right)^k \notag\\
       &=\begin{cases}
         \frac{1-p}{M^{d-2}(M^d p-1)}\cdot p^{4n}\left(1-(M^{d}p)^{-n}\right), & M^dp\neq 1, \\ \frac{1-p}{M^{2(d-1)}}\cdot p^{4n-1} \cdot n, & M^dp= 1.
      \end{cases} \label{eq:3rd-term-type2-3}
     \end{align}

     \paragraph{\bf Analysis of the last term in \eqref{eq:sur-var-proof1}} First note that the covariances $\Cov(X_n^{j},X_n^{i,\ell})$ in the last term vanish whenever $j\notin\{i,\ell\}$, because the variables $X_n^{j}$ and $X_n^{i,\ell}$ are determined by disjoint subfamilies of the basic random variables $Y_\sigma$, $\sigma\in\Sigma^*$ in this case and are therefore independent. Hence, all relevant terms in the last sum refer to pairs $(j, \{i,\ell\})$ such that $J_i$ and $J_\ell$ are direct neighbors and $j\in\{i,\ell\}$. Let us first determine the number of such terms. For each couple $\{i,l\}$ there are two choices for $j$ (namely, $j=i$ and $j=\ell$), and there are $d(M-1)M^{d-1}$ couples such that $J_i$ and $J_\ell$ are direct neighbors. Hence there are $2d(M-1)M^{d-1}$ such terms.
      To find an expression for this covariance, we choose $j=i$ without loss of generality. Again we start by separating the influence of the first construction step from the rest. Setting $\widehat{X}_n:=V_{d-1}(F_{2,n}^i\cap F_{2,n}^\ell)$ and $\check{X}_n:=V_{d-1}(F_{2,n}^i)$ we infer, using the first formula in Lemma~\ref{lem:cov-basics},
  \begin{align*}
     \Cov(X_n^{i},X_n^{i,\ell}) 
     &=\Cov\left(V_{d-1}(F_{2,n}^i)\cdot Y_i,V_{d-1}(F_{2,n}^i\cap F_{2,n}^{\ell})\cdot Y_iY_\ell\right)\\
     &=\Cov\left(\check{X}_n,\widehat{X}_n\right) \E\left(Y_i^2Y_\ell\right)+ \E V_{d-1}(F_{2,n}^i)\cdot \E \widehat{X}_n\cdot \Cov(Y_i, Y_i Y_\ell)\\
     &=\Cov\left(\check{X}_n,\widehat{X}_n\right)\cdot p^2 + \E V_{d-1}(F_{2,n}^i) \cdot \E \widehat{X}_n\cdot p^2(1-p).
     \end{align*}
    Now recall from \eqref{eq:exp-Xn-Xnil2} that $\E \widehat{X}_n\cdot p^2=\E X_n^{i,\ell}$, for which an explicit expression is provided in \eqref{eq:exp-Xnil}. Moreover, by \eqref{eq:1-lem2-1}, we have
    $\E \check{X}_n=M^{1-d} \E V_{d-1}(F_{n-1})$ and for the latter an expression is given by \eqref{eq:exp-surf-Exp}. Hence, setting $\gamma_n:=\Cov \left(\check{X}_{n},\widehat{X}_{n}\right)$, $n\in\N$, (and noting that $\gamma_1=0$), we infer that, for any $n\in\N$,
   \begin{align} \label{eq:4th-term-expr1}
     \Cov(X_n^{i},X_n^{i,\ell})
     &= 
     \gamma_n \cdot p^2 + (Mp^3)^{n} \frac {d(1-p)^2}{M^{2(d-1)}p(M-p)} \left(1+\frac{M-1}{1-p}\left(\frac pM\right)^{n}\right).
   \end{align}
     Also this time the covariance term $\gamma_n$ on the right does not vanish. But we can derive a recursion for $\gamma_n$ with a similar approach as the one used for the second type in the third term.
     The intersection $J_i\cap J_\ell$ is $(d-1)$-dimensional and intersects $M^{d-1}$ of the second level cubes $J_{ij}$ contained in $J_i$. Denote those latter cubes by $Q_1, Q_2,\ldots, Q_{M^{d-1}}$ and let $R$ be the union of all the remaining second level cubes contained in $J_i$, i.e. $R=\cup\{J_{ij}:j\in\Sigma, J_{ij}\neq Q_m $ for all $m\}$. Observe that $V_{d-1}(F_{2,n}^i\cap R\cap F_{2,n}^\ell)=0$ and therefore
     \begin{align*}
       \widehat{X}_n= V_{d-1}(F_{2,n}^i\cap F_{2,n}^\ell)=\sum_{k=1}^{M^{d-1}} \underbrace{V_{d-1}(F_{2,n}^i\cap Q_k\cap F_{2,n}^\ell)}_{=:\widehat{Q}_k},
     \end{align*}
     where the terms on the right are independent. What happens in $Q_k$ is now again a scaled version of what happens in $J_i$. More precisely, we have, for $k\in\{1,\ldots, M^{d-1}\}$,
     \begin{align*} 
        \widehat{Q}_k & 
        \deq M^{1-d}V_{d-1}(F_{n-1}^i\cap F_{n-1}^\ell)=M^{1-d} X_{n-1}^{i,\ell}.
     \end{align*}
     Observe that, for any $n\in\N$,
     \begin{align} \label{eq:sur-proof-4th-term}
        \gamma_n&= 
        \Cov\left(\check{X}_{n},\sum_{k=1}^{M^{d-1}} \widehat{Q}_k\right)
        =\sum_{k=1}^{M^{d-1}} \Cov\left(\check{X}_{n},\widehat{Q}_k\right).
     \end{align}
     Now we split $\check{X}_{n}=V_{d-1}(F_{2,n}^i)$ in a similar way using the decomposition $J_i=R\cup\bigcup_k Q_k$. It is clear that the main contribution of the $k$-th term above will come from $\Cov(V_{d-1}(F_{2,n}^i\cap Q_k),\widehat{Q}_k)$ but there are further contributions this time which are due to the fact that the sets $F_{2,n}^i\cap Q_j$ are full-dimensional such that their intersections cannot be neglected. In fact, in order to separate the contributions of the cubes $Q_k$ we need a more refined notation. Without loss of generality, we can assume that the basic cubes $J_{ij}$ contained in $J_i$ are enumerated in such a way that $J_{ik}=Q_k$ for $k=1,\ldots,M^{d-1}$. Let $F_{2,n}^{ij}$ be the union of those cubes $J_\sigma$ of level $n$ (i.e.\ $\sigma=\sigma_1\ldots\sigma_n\in\Sigma^n$) contained in $F_{2,n}$ such that $\sigma_1=i$ and $\sigma_2=j$, $j\in\Sigma$, and let $F_{2,n}^{i,R}:=\bigcup_{j>M^{d-1}} F_{2,n}^{ij}$. By the inclusion-exclusion principle and ignoring lower-dimensional intersections, we have
     \begin{align*}
       \check{X}_{n}=V_{d-1}(F_{2,n}^{i,R})+ \sum_{j=1}^{M^{d-1}} \bigg[ &V_{d-1}(F_{2,n}^{ij})- V_{d-1}(F_{2,n}^{ij}\cap F_{2,n}^{i,R})\bigg]\\ &-\sum_{\{j,m\}:Q_j\sim Q_m} V_{d-1}(F_{2,n}^{ij}\cap F_{2,n}^{im}),
     \end{align*}
     where $Q_j\sim Q_m$ in the last sum means that $Q_j$ and $Q_m$ are direct neighbors. Now we insert this representation into the $k$-th covariance term in \eqref{eq:sur-proof-4th-term} and sort out, which of the terms still have some of the basic random variables in common with $\widehat{Q}_k$ and which are independent. First observe that, for any $k$, $\Cov(V_{d-1}(F_{2,n}^{i,R}),\widehat{Q}_k)=0$, since the variables are determined by disjoint families of the basic random variables of the process. For a similar reason, the covariances $\Cov(V_{d-1}(F_{2,n}^{ij}),\widehat{Q}_k)$ and $\Cov(V_{d-1}(F_{2,n}^{ij}\cap F_{2,n}^{i,R}),\widehat{Q}_k)$ vanish whenever $j\neq k$. Finally, $\Cov(V_{d-1}(F_{2,n}^{ij}\cap F_{2,n}^{im}),\widehat{Q}_k)=0$, whenever $k\notin\{j,m\}$. Hence, we get
     \begin{align} \label{eq:4th-term-eq2}
        \Cov\left(\check{X}_{n},\widehat{Q}_k\right)
        &=\Cov(V_{d-1}(F_{2,n}^{ik}),\widehat{Q}_k)-\Cov(V_{d-1}(F_{2,n}^{ik}\cap F_{2,n}^{i,R}),\widehat{Q}_k)\notag\\&\qquad - \sum_{m:Q_m\sim Q_k} \Cov(V_{d-1}(F_{2,n}^{ik}\cap F_{2,n}^{im}),\widehat{Q}_k).
     \end{align}
In order to derive a recursion, note that analogously to Lemma~\ref{lem:self-sim-j}, we have
\begin{align*}
   &\left(V_{d-1}(F_{3,n}^{ik}),V_{d-1}(F_{3,n}^i\cap Q_k\cap F_{2,n}^\ell)\right)\\
   &\qquad\deq M^{1-d} \left(V_{d-1}(F_{2,n-1}^i), V_{d-1}(F_{2,n-1}^i\cap F_{2,n-1}^\ell)\right)=M^{1-d} \left(\check{X}_{n-1},\widehat{X}_{n-1}\right),
\end{align*}
which implies for the first term in the right hand side of \eqref{eq:4th-term-eq2}
\begin{align*}
  \Cov&(V_{d-1}(F_{2,n}^{ik}),\widehat{Q}_k) = \Cov(V_{d-1}(F_{3,n}^{ik})\cdot Y_{ik},V_{d-1}(F_{3,n}^{i}\cap Q_k\cap F_{2,n}^\ell)\cdot Y_{ik})\\
  &=\Cov(V_{d-1}(F_{3,n}^{ik}),V_{d-1}(F_{3,n}^{i}\cap Q_k\cap F_{2,n}^\ell))\cdot \E(Y_{ik}^2)\\
  &\qquad +\E V_{d-1}(F_{3,n}^{ik}) \E V_{d-1}(F_{3,n}^{i}\cap Q_k\cap F_{2,n}^\ell)\Var(Y_{ik})\\
  &= M^{2(1-d)} \Cov\left(\check{X}_{n-1},\widehat{X}_{n-1}\right)\cdot p
  +M^{2(1-d)} \E\check{X}_{n-1}\cdot  \E \widehat{X}_{n-1}\cdot p(1-p)\\
  &=\frac p{M^{2(d-1)}} \cdot \gamma_{n-1} +\frac{1-p}{M^{3(d-1)}} p^{2n-3} \cdot \E\check{X}_{n-1}.
\end{align*}
Here we have used \eqref{eq:exp-Xn-Xnil2} and \eqref{eq:exp-Xnil} to replace $\E \widehat{X}_{n-1}$ in the last equality. Similarly, employing \eqref{eq:self-sim-j-Exp} and \eqref{eq:exp-surf-Exp} of Theorem~\ref{thm:var-surf}, we can replace $\E\check{X}_{n-1}$ with some explicit expression.
We conclude, that
\begin{align}
  \Cov(V_{d-1}&(F_{2,n}^{ik}),\widehat{Q}_k)
  =\frac p{M^{2(d-1)}} \cdot \gamma_{n-1} +\frac{1-p}{M^{3(d-1)}} p^{2n-2} \cdot \E V_{d-1}(F_{n-2}) \notag \\
  &=\frac p{M^{2(d-1)}}  \gamma_{n-1} +
  (Mp^3)^{n} \frac {d(1-p)^2}{M^{4d-3}p^4(M-p)} \left(1+\frac{M-1}{1-p}\left(\frac pM\right)^{n-1}\right). \label{eq:4rd-term-recursion-part}
\end{align}
Inserting this into \eqref{eq:4th-term-eq2} and combining it with \eqref{eq:sur-proof-4th-term} yields the desired recursion relation for $\gamma_n$, provided we can compute the two missing terms in \eqref{eq:4th-term-eq2}. For the second term in \eqref{eq:4th-term-eq2}, $\Cov(V_{d-1}(F_{2,n}^{ik}\cap F_{2,n}^{i,R}),\widehat{Q}_k)$, 
observe that we can replace $F_{2,n}^{i,R}$ by $F_{2,n}^{ik'}$ without changing the random variable $V_{d-1}(F_{2,n}^{ik}\cap F_{2,n}^{i,R})$, where $k'$ is the unique index such that $J_{ik'}$ is the unique direct neighbor of $J_{ik}$ contained in $R$. Indeed, all the other cubes in $R$ have lower dimensional intersections with $F_{2,n}^{ik}$. Similarly, we can replace in $\widehat{Q}_k=V_{d-1}(F_{2,n}^i\cap Q_k\cap F_{2,n}^\ell)$, the set $F_{2,n}^i\cap Q_k$ by $F_{2,n}^{ik}$ and the set $F_{2,n}^\ell$ by $F_{2,n}^{\ell \ell'}$, where $\ell'$ is the unique index such that the cubes $J_{ik}$ and $J_{\ell \ell'}$ are direct neighbors. Hence,
\begin{align*}
   \Cov(V_{d-1}(F_{2,n}^{ik}\cap F_{2,n}^{i,R}),\widehat{Q}_k)&=\Cov(V_{d-1}(F_{2,n}^{ik}\cap F_{2,n}^{ik'}),V_{d-1}(F_{2,n}^{ik}\cap F_{2,n}^{\ell \ell'})).
\end{align*}
Note that the cubes $J_{ik'}$, $J_{ik}$ and $J_{\ell \ell'}$ lie in a row, and therefore we are in the situation of the first type of covariance studied for the third term (cf.\ \eqref{eq:3rd-term-in-a-row} and the paragraph before this equation), but one level further down. More precisely, we have
\begin{align*}
  \left(V_{d-1}(F_{2,n}^{ik}\cap F_{2,n}^{ik'}),\right.&\left.V_{d-1}(F_{2,n}^{ik}\cap F_{2,n}^{\ell \ell'})\right)\\&\deq M^{1-d} \left(V_{d-1}(F_{n-1}^{i}\cap F_{n-1}^{i'}),V_{d-1}(F_{n-1}^{i}\cap F_{n-1}^{\ell})\right),
\end{align*}
where $i'$ is the (unique) index such that $J_{i'}$, $J_i$ and $J_\ell$ lie in a row. (Note that $i'$ might not exist if $J_i$ touches the boundary of $J$, but then one can shift the situation or think of $F_{n-1}^{i'}$ as being part of a second fractal percolation in the corresponding unit size square neighboring $J$. This will not change the stated distributional relation.) Employing \eqref{eq:3rd-term-in-a-row}, we conclude that for any $n\in\N$, $n\geq 2$,
\begin{align} \label{eq:4rd-term-in-a-row}
     \Cov(V_{d-1}(F_{2,n}^{ik}\cap F_{2,n}^{i,R}),\widehat{Q}_k)
     &=M^{2(1-d)}\Cov(X_{n-1}^{i,i'},X_{n-1}^{i,\ell})=\frac {1-p}{M^{4(d-1)}}p^{4n-5},
     \end{align}
     while this covariance vanishes for $n\leq 1$.

The last sum in \eqref{eq:4th-term-eq2} consists of $2(d-1)$ terms of the form $\Cov(V_{d-1}(F_{2,n}^{ik}\cap F_{2,n}^{im}),\widehat{Q}_k)$, for which the corresponding cubes $J_{ik}$ and $J_{im}$ are direct neighbors. As for the previous term we can use the relation $\widehat{Q}_k=V_{d-1}(F_{2,n}^{ik}\cap F_{2,n}^{\ell\ell'})$, where $\ell'$ was the unique index such that the cubes $J_{ik}$ and $J_{\ell \ell'}$ are direct neighbors. Hence,
\begin{align*}
   \Cov(V_{d-1}(F_{2,n}^{ik}\cap F_{2,n}^{im}),\widehat{Q}_k)&=\Cov(V_{d-1}(F_{2,n}^{ik}\cap F_{2,n}^{im}),V_{d-1}(F_{2,n}^{ik}\cap F_{2,n}^{\ell \ell'})),
\end{align*}
where now the cubes $J_{im}$ and $J_{\ell\ell'}$ have a nonempty intersection. Thus we are in the situation of the second type of covariance studied for the second term (cf.\ \eqref{eq:3rd-term-type2} and the paragraph before this equation), just one level further down. More precisely, we have
\begin{align*}
  \left(V_{d-1}(F_{2,n}^{ik}\cap F_{2,n}^{im}),\right.&\left.V_{d-1}(F_{2,n}^{ik}\cap F_{2,n}^{\ell \ell'})\right)\\
  &\deq M^{1-d} \left(V_{d-1}(F_{n-1}^{i}\cap F_{n-1}^{i'}),V_{d-1}(F_{n-1}^{i}\cap F_{n-1}^{\ell})\right),
\end{align*}
where now  $i'$ is one of the indices such that $J_{i'}$ and $J_i$ are direct neighbors and $J_{i'}$ and $J_\ell$ are distinct and have a nonempty intersection. Employing \eqref{eq:3rd-term-type2}, we conclude that for $p>M^{-d}$ and any $n\in\N$, $n\geq 2$,
\begin{align}
     \sum_{m: Q_k\sim Q_m}\Cov(&V_{d-1}(F_{2,n}^{ik}\cap F_{2,n}^{im}),\widehat{Q}_k)
     =2(d-1)M^{2(1-d)}\Cov(X_{n-1}^{i,i'},X_{n-1}^{i,\ell}) \notag\\
     &=
         \frac{2(d-1)(1-p)}{M^{3d-4}(M^d p-1)}\cdot p^{4(n-1)}\left(1-(M^{d}p)^{-n+1}\right),
     \label{eq:4rd-term-type2} 
     \end{align}
     while this covariance vanishes for $n\leq 1$.

    Plugging now the expressions derived in \eqref{eq:4rd-term-recursion-part}, \eqref{eq:4rd-term-in-a-row} and \eqref{eq:4rd-term-type2} into \eqref{eq:4th-term-eq2} and summing over all $k=1,...,M^{d-1}$, we obtain the desired recursion for the covariance $\gamma_n$ in \eqref{eq:sur-proof-4th-term}. For $n\geq 2$ and $M^dp\neq 1$, we have
    \begin{align*}
       \gamma_n 
       &=\frac p{M^{d-1}} \cdot \gamma_{n-1} +
  (Mp^3)^{n} \frac {d(1-p)^2}{M^{3d-2}(M-p)p^4} \left(1+\frac{M-1}{1-p}\left(\frac pM\right)^{n-1}\right)\\
  &\qquad -\frac {1-p}{M^{3(d-1)}}p^{4n-5}-\frac{2(d-1)(1-p)}{M^{2d-3}(M^d p-1)}\cdot p^{4(n-1)}\left(1-(M^{d}p)^{-n+1}\right)\\
  &=\frac p{M^{d-1}} \cdot \gamma_{n-1} +
  \tilde c_1\cdot (Mp^3)^{n} +\tilde c_2\cdot p^{4n} + \tilde c_3 \cdot \left(M^{-d}p^3\right)^n,
  \end{align*}
    where $\tilde c_1:=\frac {d(1-p)^2}{M^{3d-2}p^4(M-p)} $, $\tilde c_3:=\frac{2(d-1)(1-p)}{M^{d-3}p^3(M^d p-1)}$ and
    \begin{align*}
       \tilde c_2&:=\frac{d(M-1)(1-p)}{M^{3(d-1)}p^5(M-p)}-\frac {1-p}{M^{3(d-1)}p^5}-\frac{2(d-1)(1-p)}{M^{2d-3}(M^d p-1)p^4}\\
       &=\frac{1-p}{M^{3(d-1)}p^5}\left(d\frac{M-1}{M-p}-1-2(d-1)\frac{M^dp}{M^dp-1}\right).
    \end{align*}
    By induction, we get from the above recursion that, for $n\geq 2$,
    \begin{align*}
      \gamma_n 
      =& (Mp^3)^n \tilde c_1\frac{M^dp^2}{M^dp^2-1} +p^{4n} \tilde c_2 \frac{M^{d-1}p^3}{M^{d-1}p^3-1} +\left(\frac{p^{3}}{M^d}\right)^{n} \tilde c_3 \frac{p^2}{p^2-M}\\
      &-\left(\frac{p}{M^{d-1}}\right)^{n}\underbrace{\left(\tilde c_1 \frac{(M^{d}p^2)^2}{M^dp^2-1}+ \tilde c_2 \frac{(M^{d-1}p^3)^2}{M^{d-1}p^3-1}+ \tilde c_3\frac{(M^{-1}p^2)^2}{M^{-1}p^2-1}\right)}_{=:\tilde c_4},
    \end{align*}
     where we have used that  $\gamma_1=0$ and where we have excluded the case $p^2=M^{-d}$, in which the constant in the first term has to be replaced by $\tilde c_1(n-1)$ and the first summand in $\tilde c_4$ by $0$.
     In the sequel, we will only need an asymptotic expression for $\gamma_n$ as $n\to\infty$. It turns out that it depends on the parameter combination, which term is the leading one.
    For $p>M^{-d}$, we obtain, as $n\to\infty$,
     \begin{align*}
       \gamma_n
  &=\begin{cases}
     \tilde c_1\cdot \frac{M^{d}p^2}{M^{d}p^2-1} (Mp^3)^{n}+ O(\max\{p^{4}, p/M^{d-1}\}^n), & \text{ if } p>M^{-d/2},\\
     \tilde c_1 \cdot (Mp^3)^{n} \cdot (n-1)+ O((Mp^3)^{n}),& \text{ if } p^2= M^{-d},\\
     \tilde c_4 \cdot (p/M^{d-1})^n + O((Mp^3)^n),& \text{ if } p< M^{-d/2}.
  \end{cases}
  \end{align*}
  Now we can complete the analysis of the 4-th term. Inserting the last expression into equation \eqref{eq:4th-term-expr1} and noting that the second term in \eqref{eq:4th-term-expr1} is always of order $O((Mp^{3})^n)$, we conclude that for any $p>M^{-d}$, as $n\to\infty$,
  \begin{align} \label{eq:cov-term4-final} 
     \Cov(X_n^{i},X_n^{i,\ell})&=  \begin{cases}
     \hat c_1\cdot(Mp^3)^{n}+ O(\max\{p^{4}, p/M^{d-1}\}^n),
     & \text{if } p>M^{-d/2},\\
     \tilde c_1 \cdot (Mp^3)^{n} \cdot (n-1)+ O((Mp^3)^{n}),& \text{if } p^2= M^{-d},\\
     \tilde c_4 \cdot (p/M^{d-1})^n + O((Mp^3)^n),& \text{if } p< M^{-d/2},
  \end{cases}
    %
        \end{align}
        where $ \hat c_1:=\frac{d(1-p)}{M^{2(d-1)}p(M-p)} (1+\frac{1-p}p\frac{1}{M^dp^2-1})$.
  \paragraph{\bf Computation of the variance of $X_n=V_{d-1}(F_n)$.}
  Now that we have determined all the terms on the right hand side of \eqref{eq:sur-var-proof1}, we can go back to this equation and put all the pieces together. Setting $\beta_n:=\Var(X_n)$ and combining \eqref{eq:sur-var-recursion1} with  \eqref{eq:sur-var-proof1}, we get
   \begin{align} \label{eq:sur-var-proof2}
    \notag \beta_n    &= \frac{p}{M^{d-2}}\beta_{n-1}
    + \frac {d^2(1-p)^3}{M^{d-2}p(M-p)^2} (Mp)^{2n}
    \left(1+ \frac{M-1}{1-p}\left(\frac pM\right)^n\right)^2\\
    & \quad + \sum_{\{i,\ell\}\subset\Sigma}\hspace{-2mm} \Var(X_n^{i,\ell})+ \sum_{
    \{i,\ell\}\neq\{i',\ell'\}}\hspace{-4mm} \Cov(X_n^{i,\ell},X_n^{i',\ell'}) -2\sum_{j\in\Sigma}\sum_{\{i,\ell\}\subset\Sigma}\Cov\left(X_n^j,X_n^{i,\ell}\right).
  \end{align}
 Observe from \eqref{eq:var-Xnjl-3}, that the first sum in the second line above satisfies
  \begin{align*}
     \sum_{\{i,\ell\}\subset\Sigma} \Var(X_n^{i,\ell})&=  \begin{cases}
 O(p^{4n}), &M^{d-1}p^2>1,\\
  O(p^{4n}\cdot n), &M^{d-1}p^2= 1,\\
  O((p^2/M^{d-1})^n), &M^{d-1}p^2< 1,
     \end{cases}
  \end{align*}
  as $n\to\infty$. So the order of the leading term depends on whether $p$ is below,  at or above the value $M^{-\frac{d-1}2}$. In any case this term is certainly dominated by $p^{2n}$, as $n\to\infty$, (i.e., $\sum_{\{i,\ell\}\subset\Sigma} \Var(X_n^{i,\ell})=o(p^{2n})$) and vanishes thus rapidly for $n\to\infty$.
  Equations \eqref{eq:3rd-term-in-a-row} and \eqref{eq:3rd-term-type2-3} show that  the situation for the second sum in \eqref{eq:sur-var-proof2} is easier. For  $p>M^{-d}$ and any $n\in\N$, we have
 \begin{align*}
    \sum_{
    \{i,\ell\}\neq\{i',\ell'\}} \Cov(X_n^{i,\ell},X_n^{i',\ell'})
    &=\frac {2d(M-2)(1-p)}{M^{d-1}p}p^{4n}
    + 8\binom{d}{2}(M-1)^2\times \\ &\quad 
         \frac{1-p}{M^d p-1}\cdot p^{4n}\left(1-(M^{d}p)^{-n}\right),
 \end{align*}
 and thus, as $n\to\infty$,
 \begin{align*}
    \sum_{
    \{i,\ell\}\neq\{i',\ell'\}} \Cov(X_n^{i,\ell},X_n^{i',\ell'})
    &=
          O(p^{4n}). 
 \end{align*}
 For the last sum in \eqref{eq:sur-var-proof2} we infer from \eqref{eq:cov-term4-final}, that in all three cases this term is certainly dominated by $(Mp)^{2n}$, as $n\to\infty$. Thus, the second term in \eqref{eq:sur-var-proof2} (of order $(Mp)^{2n}$) is always the leading one.  But we also need the correct second order term in order to say something about the speed of convergence of $\Var(Z_n^{d-1})$. The recursion for $\beta_n$ will produce a second term of order  $(p/M^{d-2})^n$, but it turns out that for some parameters there are others of higher order which we will therefore have to take into account in the recursion. Note that for $p^2\geq1/M^{d-1}$, we have $Mp^3\geq p/M^{d-2}$ and thus we include the corresponding term of order $(Mp^3)^n$ into the recursion, which is the one of highest order among the remaining ones for these parameters. This is not necessary for $p^2<1/M^{d-1}$. 
 More precisely, we have for any $p>M^{-d}$, as $n\to\infty$,
\begin{align*}
  \beta_n=\frac{p}{M^{d-2}}\cdot \beta_{n-1} + \tilde{c}\cdot(Mp)^{2n}+\begin{cases} \hat c_1\cdot(Mp^3)^{n}+O(q^n), &\text{if } p^2\geq M^{-(d-1)},\\
  O((Mp^3)^n),& \text{if } p^2< M^{-(d-1)},
  \end{cases}
\end{align*}
where $\tilde{c}:=\frac {d^2(1-p)^3}{M^{d-2}p(M-p)^2}$, $\hat c_1$ is the same as in \eqref{eq:cov-term4-final} and $q$ is some number strictly less than $Mp^3$ (depending on the parameter combination).
By induction, this yields
\begin{align*}
  \beta_n=\tilde{c} \frac{M^{d}p}{M^{d}p-1}(Mp)^{2n} +
  \begin{cases}
    \hat c_1 \frac{M^{d-1}p^2}{M^{d-1}p^2-1}(Mp^3)^n+ O(s^n),& \text{if } p^2>1/M^{d-1},\\
    \hat c_1 (Mp^3)^n \cdot n+O((Mp^3)^n), & \text{if } p^2=1/M^{d-1},\\
    \bar c \left(\frac{p}{M^{d-2}}\right)^n +O((Mp^3)^n) , & \text{if } p^2<1/M^{d-1},\\
  \end{cases} 
\end{align*}
as $n\to\infty$, where $\tilde c$ and $\hat c_1$ are as above, $s<Mp^3$ is some number and $\bar c\in\R$ is some constant, we do not give explicitly here, as it would require to determine all terms in the above recursion exactly. Here we have used that $\beta_0=\Var(V_{d-1}(F_0))=0$ and that $p>M^{-d}$ is equivalent to $M^2p^2>p/M^{d-2}$. Note that $\tilde{c} \frac{M^{d}p}{M^{d}p-1}=\bar c_2$, i.e.\ the constant in first order term of $\Var(V_{d-1}(F_n))$ is as asserted in \eqref{eq:exp-surf-Var}, and also the second order terms given in \eqref{eq:exp-surf-Var} are transparent from the above equation.
This completes the proof of the formula for the variance in Theorem~\ref{thm:var-surf}.
     \end{proof}

 \fund 
\noindent S.W.~was supported by the Deutsche Forschungsgemeinschaft
(DFG, German Research Foundation) grant WI 3264/5-1.
M.A.K.~was supported by the DFG through the SPP 2265 under grant numbers
LO 418/25-1, WI 5527/1-1, and ME 1361/16-1.
M.A.K.~also acknowledges funding by the Volkswagenstiftung via the
Experiment-Projekt Mecke and by the Helmholtz Association and the DLR via
the Helmholtz Young Investigator Group ``DataMat''.

\competing 
\noindent There were no competing interests to declare which arose during the preparation or publication process of this article.

\bibliographystyle{abbrv}
\footnotesize
\bibliography{mp}

\begin{thebibliography}{10}

\bibitem{MR2928497}
I.~Arhosalo, E.~J\"{a}rvenp\"{a}\"{a}, M.~J\"{a}rvenp\"{a}\"{a}, M.~Rams, and
  P.~Shmerkin.
\newblock Visible parts of fractal percolation.
\newblock {\em Proc. Edinb. Math. Soc. (2)}, 55(2):311--331, 2012.

\bibitem{AN72}
K.~B. Athreya and P.~E. Ney.
\newblock {\em Branching processes}.
\newblock Springer-Verlag, New York-Heidelberg, 1972.
\newblock Die Grundlehren der mathematischen Wissenschaften, Band 196.

\bibitem{BJ19}
A.~Berlinkov and E.~J\"{a}rvenp\"{a}\"{a}.
\newblock Porosities of {M}andelbrot percolation.
\newblock {\em J. Theoret. Probab.}, 32(2):608--632, 2019.

\bibitem{biggins77}
J.~D. Biggins.
\newblock Martingale convergence in the branching random walk.
\newblock {\em J. Appl. Probability}, 14(1):25--37, 1977.

\bibitem{BroCam08}
E.~I. Broman and F.~Camia.
\newblock {Large-{$N$} limit of crossing probabilities, discontinuity, and
  asymptotic behavior of threshold values in {M}andelbrot's fractal percolation
  process}.
\newblock {\em Electron. J. Probab.}, 13:no. 33, 980--999, 2008.

\bibitem{BroCam10}
E.~I. Broman and F.~Camia.
\newblock {Universal behavior of connectivity properties in fractal percolation
  models}.
\newblock {\em Electron. J. Probab.}, 15:1394--1414, 2010.

\bibitem{BCJM13}
E.~I. Broman, F.~Camia, M.~Joosten, and R.~Meester.
\newblock Dimension (in)equalities and {H}\"{o}lder continuous curves in
  fractal percolation.
\newblock {\em J. Theoret. Probab.}, 26(3):836--854, 2013.

\bibitem{MR4291467}
Z.~Buczolich, E.~J\"{a}rvenp\"{a}\"{a}, M.~J\"{a}rvenp\"{a}\"{a}, T.~Keleti,
  and T.~P\"{o}yht\"{a}ri.
\newblock Fractal percolation is unrectifiable.
\newblock {\em Adv. Math.}, 390:Paper No. 107906, 61, 2021.

\bibitem{CC89}
J.~T. Chayes and L.~Chayes.
\newblock The large-{$N$} limit of the threshold values in {M}andelbrot's
  fractal percolation process.
\newblock {\em J. Phys. A}, 22(11):L501--L506, 1989.

\bibitem{CCD88}
J.~T. Chayes, L.~Chayes, and R.~Durrett.
\newblock Connectivity properties of {M}andelbrot's percolation process.
\newblock {\em Probab. Theory Related Fields}, 77(3):307--324, 1988.

\bibitem{MR1378847}
L.~Chayes.
\newblock On the length of the shortest crossing in the super-critical phase of
  {M}andelbrot's percolation process.
\newblock {\em Stochastic Process. Appl.}, 61(1):25--43, 1996.

\bibitem{MR3736180}
C.~Chen, T.~Ojala, E.~Rossi, and V.~Suomala.
\newblock Fractal percolation, porosity, and dimension.
\newblock {\em J. Theoret. Probab.}, 30(4):1471--1498, 2017.

\bibitem{Don15}
H.~Don.
\newblock New methods to bound the critical probability in fractal percolation.
\newblock {\em Random Structures Algorithms}, 47(4):710--730, 2015.

\bibitem{Falconer86}
K.~J. Falconer.
\newblock Random fractals.
\newblock {\em Math. Proc. Cambridge Philos. Soc.}, 100(3):559--582, 1986.

\bibitem{Gatzouras00}
D.~Gatzouras.
\newblock Lacunarity of self-similar and stochastically self-similar sets.
\newblock {\em Trans. Amer. Math. Soc.}, 352(5):1953--1983, 2000.

\bibitem{Graf87}
S.~Graf.
\newblock Statistically self-similar fractals.
\newblock {\em Probab. Theory Related Fields}, 74(3):357--392, 1987.

\bibitem{KSM17}
M.~A. Klatt, G.~E. Schr\"oder-Turk, and K.~Mecke.
\newblock Anisotropy in finite continuum percolation: threshold estimation by
  minkowski functionals.
\newblock {\em J. Stat. Mech. Theor. Exp.}, 2017(2):023302, 2017.

\bibitem{KW18}
M.~A. Klatt and S.~Winter.
\newblock Geometric functionals of fractal percolation.
\newblock {\em Adv. in Appl. Probab.}, 52(4):1085--1126, 2020.

\bibitem{LO14}
G.~Last and E.~Ochsenreither.
\newblock Percolation on stationary tessellations: models, mean values, and
  second-order structure.
\newblock {\em Journal of Applied Probability}, 51(A):311–332, 2014.

\bibitem{Mandelbrot74}
B.~B. Mandelbrot.
\newblock Intermittent turbulence in self-similar cascades: divergence of high
  moments and dimension of the carrier.
\newblock {\em Journal of Fluid Mechanics}, 62(2):331--358, 1974.

\bibitem{MW86}
R.~D. Mauldin and S.~C. Williams.
\newblock Random recursive constructions: asymptotic geometric and topological
  properties.
\newblock {\em Trans. Amer. Math. Soc.}, 295(1):325--346, 1986.

\bibitem{NMW08}
R.~A. Neher, K.~Mecke, and H.~Wagner.
\newblock Topological estimation of percolation thresholds.
\newblock {\em J. Stat. Mech. Theor. Exp.}, 2008(01):P01011, 2008.

\bibitem{Nerman81}
O.~Nerman.
\newblock On the convergence of supercritical general ({C}-{M}-{J}) branching
  processes.
\newblock {\em Z. Wahrsch. Verw. Gebiete}, 57(3):365--395, 1981.

\bibitem{NZ00}
M.~E.~J. Newman and R.~M. Ziff.
\newblock Efficient monte carlo algorithm and high-precision results for
  percolation.
\newblock {\em Phys. Rev. Lett.}, 85:4104--4107, Nov 2000.

\bibitem{Patzschke97}
N.~Patzschke.
\newblock The strong open set condition in the random case.
\newblock {\em Proc. Amer. Math. Soc.}, 125(7):2119--2125, 1997.

\bibitem{MR3163542}
M.~Rams and K.~Simon.
\newblock The dimension of projections of fractal percolations.
\newblock {\em J. Stat. Phys.}, 154(3):633--655, 2014.

\bibitem{MR3316924}
M.~Rams and K.~Simon.
\newblock Projections of fractal percolations.
\newblock {\em Ergodic Theory Dynam. Systems}, 35(2):530--545, 2015.

\bibitem{Schneider14}
R.~Schneider.
\newblock {\em Convex bodies: the {B}runn-{M}inkowski theory}, volume 151 of
  {\em Encyclopedia of Mathematics and its Applications}.
\newblock Cambridge University Press, Cambridge, expanded edition, 2014.

\bibitem{SchneiderWeil08}
R.~Schneider and W.~Weil.
\newblock {\em Stochastic and integral geometry}.
\newblock Probability and its Applications (New York). Springer-Verlag, Berlin,
  2008.

\bibitem{Z11}
M.~Z{\"a}hle.
\newblock Lipschitz-{K}illing curvatures of self-similar random fractals.
\newblock {\em Trans. Amer. Math. Soc.}, 363(5):2663--2684, 2011.

\end{thebibliography}

\end{document}